\documentclass[a4paper,11pt,reqno]{amsart}
\usepackage[margin=3cm]{geometry}
\usepackage{bm}

\newcommand{\version}{\today}

\usepackage{amsthm,amsfonts,amsmath, amscd, bbold}
\usepackage{mathrsfs}
\usepackage{accents}

\usepackage{mathrsfs}%,dsfont}
\usepackage{amscd}
\usepackage[active]{srcltx}
\usepackage{verbatim}
\usepackage[colorlinks,linkcolor={black},citecolor={black},urlcolor={black}]{hyperref}

\usepackage[latin1]{inputenc}
\usepackage{tikz}
\usetikzlibrary{trees}

\usepackage{mathtools}

\swapnumbers
                              
\theoremstyle{plain}
\newtheorem{thm}{THEOREM}[section]
\newtheorem{lm}[thm]{LEMMA}

\theoremstyle{definition}
\newtheorem{defi}[thm]{DEFINITION}

\theoremstyle{remark}
\newtheorem{remark}[thm]{REMARK}
\newtheorem{example}[thm]{EXAMPLE}
\newcommand{\upchi}{\raise1pt\hbox{$\chi$}}
\newcommand{\R}{{\mathord{\mathbb R}}}
\newcommand{\C}{{\mathord{\mathbb C}}}
\newcommand{\Z}{{\mathord{\mathbb Z}}}

\newcommand{\tr}{{\rm Tr}}
\numberwithin{equation}{section}
\newcommand{\un}{{\rm 1\kern -2.5pt l}}

 \newcommand{\rb}{\rho_{2}}

 \newcommand{\Dens}{{\mathfrak S}}

\begin{document}
%%%%%%%%DRAFT%%%%%%%
\iffalse
\fi
%%%%%%%%%%%%%%%%%%%%%%
\markboth{\scriptsize{CM \version}}{\scriptsize{CM September 2, 2016}}
\def\Z{{\mathbb Z}}
\def\R{{\mathbb R}}
\def\C{{\mathbb C}}
\def\sg{\sigma}
\def\S{\mathcal{S}}
% probability notations

\newcommand{\E}{{\mathcal E}}
\def\Ex{{\mathbb E}}

% greek letters

\renewcommand{\a}{\alpha}
\renewcommand{\b}{\beta}
\newcommand{\ga}{\gamma}
\newcommand{\Ga}{\Gamma}
\renewcommand{\d}{\delta}
\newcommand{\e}{\varepsilon}
\renewcommand{\l}{\lambda}
\newcommand{\om}{\omega}
\renewcommand{\O}{\Omega}
\newcommand{\Om}{\Omega}
\renewcommand{\th}{\theta}
\newcommand{\VV}{\mathbf{V}}
\newcommand{\UU}{\mathbf{U}}

\definecolor{jan}{rgb}{0.0,0.3,0.8}

\def\AAA{\color{blue}}   
\def\JJJ{\color{red}}
\def\MMM{\color{red}}
\def\EEE{\color{black}\normalsize}
\def\DDD{\color{magenta}\footnotesize}

\makeatletter
\DeclareRobustCommand\widecheck[1]{{\mathpalette\@widecheck{#1}}}
\def\@widecheck#1#2{%
    \setbox\z@\hbox{\m@th$#1#2$}%
    \setbox\tw@\hbox{\m@th$#1%
       \widehat{%
          \vrule\@width\z@\@height\ht\z@
          \vrule\@height\z@\@width\wd\z@}$}%
    \dp\tw@-\ht\z@
    \@tempdima\ht\z@ \advance\@tempdima2\ht\tw@ \divide\@tempdima\thr@@
    \setbox\tw@\hbox{%
       \raise\@tempdima\hbox{\scalebox{1}[-1]{\lower\@tempdima\box
\tw@}}}%
    {\ooalign{\box\tw@ \cr \box\z@}}}
\makeatother

% commands

\newcommand{\beq}{\begin{equation}}
\newcommand{\eeq}{\end{equation}}
\newcommand{\bal}{\begin{aligned}}
\newcommand{\eal}{\end{aligned}}
\newcommand{\ben}{\begin{enumerate}}
\newcommand{\beni} {\begin{enumerate}[(i)]}
\newcommand{\een}{\end{enumerate}}
\newcommand{\bit}{\begin{itemize}}
\newcommand{\eit}{\end{itemize}}
\newcommand{\beqw}{\begin{equation*}}
\newcommand{\eeqw}{\end{equation*}}
\newcommand{\bthm}{\begin{theorem}}
\newcommand{\ethm}{\end{theorem}}
\newcommand{\bpr}{\begin{proposition}}
\newcommand{\epr}{\end{proposition}}
\newcommand{\ble}{\begin{lemma}}
\newcommand{\ele}{\end{lemma}}
\newcommand{\blem}{\begin{lemma}}
\newcommand{\elem}{\end{lemma}}
\newcommand{\bpf}{\begin{proof}}
\newcommand{\epf}{\end{proof}}
\newcommand{\bex}{\begin{example}}
\newcommand{\eex}{\end{example}}
\newcommand{\bre}{\begin{example}}
\newcommand{\ere}{\end{example}}
\newcommand{\bma}{\begin{bmatrix}}
\newcommand{\ema}{\end{bmatrix}}
\newcommand\T{{\mathbb T}}

\newcommand{\ddt}{\frac{\mathrm{d}}{\mathrm{d}t}}
\newcommand{\ddtr}{\frac{\mathrm{d}^+}{\mathrm{d}t}}
\newcommand{\ddhr}{\frac{\mathrm{d}^+}{\mathrm{d}h}}
\newcommand{\ddtt}{\frac{\mathrm{d^2}}{\mathrm{d}t^2}}
\newcommand{\ddr}{\frac{\mathrm{d}}{\mathrm{d}r}}

% miscellaneous

\renewcommand{\aa}{{\boldsymbol\alpha}}
\newcommand{\Dom}{{\mathsf D}}
\newcommand{\wt}{\widetilde}
\newcommand{\one}{{{\bf 1}}}
\newcommand{\embed}{\hookrightarrow}
\newcommand{\Null}{\mathsf{N}}
\newcommand{\rank}{\mathrm{rank}\;}
\renewcommand{\Re}{{\rm{Re}}\;}
\renewcommand{\Im}{{\rm{Im}}\;}
\newcommand{\sgn}{{\rm sgn}}
\renewcommand{\hat}{\widehat}
\newcommand{\ot}{\otimes}
\newcommand{\hN}{\widehat\Num}
\newcommand{\G}{\Gamma}
\renewcommand{\phi}{\varphi}
\def\bh{{\bf h}}
\def\tand{\quad{\rm and}\quad}
\def\L{\mathcal{L}}\def\bk{\rangle\langle}
\newcommand{\cA}{{\mathord{\mathscr A}}}
\def\ib{\iota_\beta}
\def\H{\mathcal{H}}
\def\E{\mathcal{E}}
\def\Eb{\mathcal{E}_\beta}
\def\rb{\sigma_\beta}
\newcommand{\bx}{{\boldsymbol{x}}}
\newcommand{\cM}{\mathcal{M}}
\newcommand{\aA}{\mathcal{A}}
\newcommand{\aB}{\mathcal{B}}
\newcommand{\cG}{\mathcal{G}}
\newcommand{\cP}{{\mathord{\mathscr P}}}
\newcommand{\cQ}{{\mathord{\mathscr Q}}}
\newcommand{\cE}{{\mathcal{E}}}
\newcommand{\cL}{{\mathord{\mathscr L}}}
\newcommand{\cK}{{\mathord{\mathscr K}}}
\newcommand{\cB}{{\mathord{\mathscr B}}}
\newcommand{\cJ}{{\mathcal{J}}}
\newcommand{\cC}{{\mathcal{C}}}

\newcommand{\B}{{\mathcal{B}}}
\newcommand{\calS}{{\mathcal{S}}}
\newcommand{\F}{{\mathcal{F}}}
\newcommand{\fH}{{\mathfrak{H}}}
\newcommand{\bbA}{{\bf A}}
\newcommand{\Cln}{\mathfrak{C}^n}  
\newcommand{\err}{{r}}

\newcommand{\PX}{\cP(\cX)}
\newcommand{\PXs}{\cP_*(\cX)}
\newcommand{\hrhom}{\widehat{\rho_{-j}}}
\newcommand{\hrhop}{\widehat{\rho_{+j}}}
\newcommand{\hrhopm}{\widehat{\rho_{\pm j}}}
\newcommand{\hrhomp}{\widehat{\rho_{\mp j}}}
\newcommand{\lrho}{\check{\rho}}
\newcommand{\ulrho}{\breve{\rho}}
\newcommand{\ulV}{\breve{V}}
\newcommand*{\bdot}[1]{%
  \accentset{\mbox{\large\bfseries .}}{#1}}

\title
[Complete positivity and self-adjointness]
{Complete positivity and self-adjointness}

\author{\'Erik Amorim}
\address{Department of Mathematics\\ 
Hill Center\\
Rutgers University\\
110 Frelinghuysen Road\\
Piscataway\\
NJ 08854-8019\\
USA}
\email{erik.amorim@math.rutgers.edu}

\author{Eric A. Carlen}
\address{Department of Mathematics\\ 
Hill Center\\
Rutgers University\\
110 Frelinghuysen Road\\
Piscataway\\
NJ 08854-8019\\
USA}
\email{carlen@math.rutgers.edu}

\begin{abstract}
We specify the structure of  completely positive operators and quantum Markov semigroup generators that are symmetric with respect to a family of inner products, also providing new information on the order strucure an extreme points in some previously studied cases.

\end{abstract}

\maketitle

\medskip
\centerline{Keywords: quantum Markov semigroup, entropy, detailed balance.}

\centerline{Subject Classification Numbers: 47C15, 81S22}
%%%%%%%%%%%%%%%%%%%%%%%%%%%%%%%%%%%%%%%%%%%%%%% 46L57, 
%%%%%%%%%%%%%%%%%%%%%%%%%%%%%%%%%%%%%%%%%%%%%%%

%%%%%%%%%%%%%%%%%%%%%%%%%%%%%%%%%%%%%%%%%%%%%%%

\maketitle

\section{Introduction}  

\subsection{The setting and notation}

Let $\cM_N(\C)$ denote the set of $N\times N$ matrices over $\C$, $N\geq 2$. Let $\Dens_+$ denote the set of non-degenerate density matrices in $\cM_N(\C)$. That is, each $\sigma\in \Dens_+$ is a positive definite $N\times N$ matrix with unit trace. The GNS inner product on $\cM_N(\C)$ is defined by
\begin{equation}\label{GNSip}
\langle B,A\rangle_{GNS} = \tr[B^*A\sigma]\ .
\end{equation}
Here we are concerned with a family of inner products on $\cM_N(\C)$ that all reduce to $\langle B,A\rangle_{GNS}$ when either $A$ or $B$ commutes with $\sigma$.
Let $\mathcal{P}[0,1]$ denote the set of probability measures on the interval $[0,1]$.  
\begin{defi} For each  $m\in\mathcal{P}[0,1]$, $\langle \cdot,\cdot\rangle_m$ denotes the inner product on $\cM_N(\C)$ given by
\begin{equation}
    \langle B,A \rangle_m = \tr[B^\ast\mathcal{M}_m(A)] \quad \text{where} \quad \mathcal{M}_m(A) = \int_0^1 \sigma^{s} A \sigma^{1-s} \; \mathrm{d} m (s)  \ .
    \label{integral_M}
\end{equation}
Notice that for each $s\in [0,1]$,
$$
\tr[B^* \sigma^{1-s}A \sigma^s] =  \tr[ ( \sigma^{(1-s)/2}B \sigma^{s/2})^*  \sigma^{(1-s)/2}A \sigma^{s/2}]\ ,
$$
and this quantity is strictly positive when $B=A \neq0$, and hence $\langle\cdot,\cdot\rangle_m$ is a non-degenerate inner product for every $m$.

 We call the measure $m$ and its associated inner product $\langle\cdot,\cdot\rangle_m$ {\em even} when $m$ is symmetric with respect to reflection about $s=1/2$:
\begin{equation}
m(U) = m(1-U) \quad \text{for all measurable } U\subseteq [0,1] \ .
\end{equation}
\end{defi}

Note that the GNS inner product corresponds to $m = \delta_0$, the point mass at $s=0$. Other cases are known by name. Taking $m = \delta_{1/2}$ yields the {\em Kubo-Martin-Schwinger} (KMS) inner product
\begin{equation}\label{KMSip}
\langle B,A\rangle_{KMS} = \tr[B^*\sigma^{1/2}A\sigma^{1/2}]\ .
\end{equation}
Taking $m$ to be uniform on $[0,1]$ yields the {\em Bogoliubov-Kubo-Mori} (BKM) inner product,
\begin{equation}\label{BKMip}
\langle B,A\rangle_{BKM} = \int_0^1\tr[B^*\sigma^s A\sigma^{1-s}]{\rm d}s\ .
\end{equation}
Unlike the GNS inner product, these two are even. 

Finally, if $\sigma = N^{-1}\one$, the normalized identity, all choices of $m$ reduce to the normalized Hilbert-Schmidt inner product.  Throughout this paper, $\mathfrak{H}$ always denotes the Hilbert space $\cM_N(\C)$ equipped with this inner product,
\begin{equation}\label{frakH}
\langle B,A\rangle_{\mathfrak{H}} = \frac1N\tr[B^*A]\ .
\end{equation}
Let  $\mathcal{L}(\cM_N(\C))$ denote the linear operators on $\cM_N(\C)$, or, what is the same thing in this finite dimensional setting, on $\mathfrak{H}$.  Throughout this paper, a dagger is always used to denote the adjoint with respect to the inner product on $\mathfrak{H}$.  That is, for $\Phi\in \mathcal{L}(\cM_N(\C))$,  $\Phi^\dagger$ is defined by
\begin{equation}\label{dagg}
\langle \Phi^\dagger(B),A\rangle_{\mathfrak{H}} =  \langle B, \Phi(A)\rangle_{\mathfrak{H}} 
\end{equation}
for all $A,B$.   We can also make $\mathcal{L}(\cM_N(\C))$ into a Hilbert space by equipping it with the normalized Hilbert-Schmidt inner product. Throughout this paper, this  Hilbert space is denoted by $\widehat{\mathfrak{H}}$. The following formula for the inner product in 
$\widehat{\mathfrak{H}}$ is often useful. Let  $\{F_{i,j}\}_{1\leq i,j\leq N}$ be any  orthonormal basis for $\mathfrak{H}$. Then for $\Phi$ and $\Psi\in \widehat{\mathfrak{H}}$, 
\begin{equation}\label{dagg2}
\langle \Psi,\Phi\rangle_{\widehat{\mathfrak{H}}} = \frac{1}{N^2}\sum_{i,j=1}^N  \langle \Psi(F_{i,j}), \Phi(F_{i,j})\rangle_{\mathfrak{H}} \ .
\end{equation}

To each $\sigma\in \Dens_+$, there corresponds the {\em modular operator} $\Delta$ on $\mathfrak{H}$ defined by 
\begin{equation}\label{modop}
\Delta A = \sigma A \sigma^{-1}\ .
\end{equation}
Let $\{u_1,\dots,u_N\}$  be an orthonormal basis of $\C^N$ consisting of eigenvectors of $\sigma$, so that for $1 \leq j \leq N$, $\sigma u_j = \lambda_j u_j$.  For $1\leq i,j \leq N$, define 
$
E_{i,j} = \sqrt{N} | u_i \rangle \langle u_j|
$, so that $\{E_{i,j}\}_{1\leq i,j\leq N}$ is an orthonormal basis of $\fH$.  Then a simple computation shows that for each $i,j$,
\begin{equation}\label{modeig2}
\Delta E_{i,j} = \lambda_i\lambda_j^{-1} E_{i,j}\ .
\end{equation}
Thus $\Delta$ is diagonalized, with positive diagonal entries, by an orthonormal basis in $\mathfrak{H}$. It follows that $\Delta$ is a positive operator on $\mathfrak{H}$.   
Using the Spectral Theorem, we may then define $\Delta^{s}$ for all $s$. This provides another way to write the operator $\cM_m$ defined in  \eqref{integral_M}:
\begin{equation}\label{integral_M2}
\cM_m = \left( \int_0^1 \Delta^s {\rm d}m\right) R_\sigma
\end{equation}
where $R_\sigma$ denotes right multiplication by $\sigma$, also a positive operator on $\mathfrak{H}$ that commutes with $\Delta$, and hence with $\int_0^1 \Delta^s {\rm d}m$. 

Finally, the terms ``Hermitian'' and ``self-adjoint'' are often used interchangeably. Here, we must make a distinction: A linear operator  $\Phi$ on $\cM_N(\C)$    is {\em Hermitian} if and only if $\Phi(A^*) = (\Phi(A))^*$ for all $A\in \cM_N(\C)$.  This is the only sense in which we shall use this term.

\subsection{The problems considered}

It is assumed that the reader is familiar with the notion of {\em completey positive} (CP) maps that was introcuced by Stinespring \cite{Sti55}, and have been much-studied since then. A concise  account containing all that  is needed here can be found in the early chapters of \cite{Paul02} or in  \cite{Sto10}.  By a quantum Markov semigroup, we mean a semigroup $\{\cP_t\}_{t\geq 0}$ of linear operators on $\cM_N(\C)$  such that each $\cP_t$ is completely positive and satisfies $\cP_t(\one) = \one$. A map  $\Phi$ on 
$\cM_N(\C)$  is {\em unital} in case $\Phi(\one) = \one$. Thus, when $\{\cP_t\}_{t\geq 0}$ is a QMS, each $\cP_t$ is unital, but unital CP maps are of wider interest in mathematical physics, and in this context 
are often referred to as {\em quantum operations}. The generator of  a QMS   $\{\cP_t\}_{t\geq 0}$  is the operator $\cL$ defined by
$$\cL := \lim_{t\to0}\frac1t(\cP_t - I) \ .$$
Note that $\cL \one = 0$, and $\cP_t = e^{t\cL}$.

We are interested here in CP maps and quantum Markov semigroups (QMS) and their generators that are self-adjoint  with respect to the inner product $\langle\cdot, \cdot\rangle_m$ for some $m\in \mathcal{P}[0,1]$ and some 
$\sigma\in \Dens_+$.   The structure of the set of QMS generators that are self-adjoint with respect to the GNS and KMS inner products has been studied  by mathematical physicists, but apart from these cases, not much is known. 
 
\begin{defi} For $m\in \mathcal{P}[0,1]$,  let $\H_m$ denote the Hilbert space obtained by equipping  $\cM_N(\C)$ with the inner product $\langle \cdot,\cdot \rangle_m$.   Let $CP_m$ denote the set of CP maps 
$\Phi$ on $\cM_N(\C)$ that are self-adjoint on $\H_m$, and let $CP_m(\one)$ denote the subset of $CP_m$ consisting of unital maps. Finally, let $QMS_m$ denote the set of QMS  generators that are self-adjoint on $\H_{m}$.  
We write $CP$ to denote the set of all CP maps (without any particular self-adjointness requirement). Likewise, we write $CP(\one)$ to denote the set of all CP unital maps and $QMS$ to denote the set of all QMS generators.
\end{defi}

As is well-known, both $CP$ and $QMS$ are convex cones.  For $CP$ this is obvious.  If $\cL_1$ and $\cL_2$ are two  generators of completely positive semigroups, then
\begin{equation}\label{trotter}
e^{t(\cL_1 + \cL_2)} = \lim_{k\to\infty} (e^{(t/k)\cL_1}e^{(t/k)\cL_2})^k
\end{equation}
is completely positive,  and hence the set of generators of completely positive semigroups is closed under addition, and also evidently under multiplication 
by non-negative real numbers.  Since if $\cL_j\one = 0$ for $j=1,2$, then $(\cL_1+\cL_2)\one =0$, it follows easily that the set of all 
$QMS$ semigroup generators is a convex cone.  However, it is not a pointed cone: Let $H\in \cM_N(\C)$ be self-adjoint, and define $\cL(A) = i[H,A]$. 
Then $e^{t\cL}(A) = e^{itH}A e^{-it H}$ is a QMS, and evidently both $\cL$ and $-\cL$  are QMS generators.

By a theorem of Lindblad \cite{Lin76} and Gorini, Kossakowski and Sudarshan \cite{GKS76}, every QMS generator  $\cL$  on $\cM_N(\C)$ is of the form
\begin{equation}\label{lind1}
\cL(A) =  (G^*A + AG) + \Phi(A) 
\end{equation}
where $\Phi$ is completely positive. Since $\cL\one = 0$,   $G^* + G = -\Phi(\one)$. If we define $K := \frac{1}{2i}(G-G^*)$, then we can rewrite \eqref{lind1} as
\begin{equation}\label{lind2}
\cL(A) =  \Phi(A) - \frac12(\Phi(\one)A + A \Phi(\one)) - i[K,A]\ .
\end{equation}
However, $\Phi$ and $K$ are not uniquely determined by $\cL$, and it is possible for $\cL$ to be self-adjoint on $\H_m$ while $\Phi$ is not, and {\em vice-versa}. 
Hence the problem of determining the structure of $CP_m$ is not the same as the problem of determining the structure of $QMS_m$.   

There is a natural order relation on $CP$ (without any requirement of self-adjointness), that was investigated by Arveson in \cite{Arv69}. He worked 
in the general context of CP maps from a $C^*$ algebra $\cA$ into $\mathcal{B}(\H)$, the bounded operators on a Hilbert space $\H$.   If $\Phi$ and $\Psi$ are two such CP maps, 
we write $\Phi \geq  \Psi$ in case $\Phi - \Psi$ is CP.  An element $\Phi$ is called {\em extreme} in case whenever $\Psi\in CP$ and  $\Phi \geq t\Psi$ for some $t>0$,  then $\Psi$ is a multiple of 
$\Phi$.  Arveson also considered the set $CP(\one)$ of unital CP maps. An element  $\Phi\in CP(\one)$ is extreme if whenever for some $0 < t < 1$ and $\Psi_1,\Psi_2\in CP(\one)$, $\Phi = t\Psi_1+ (1-t)\Psi_2$,
$\Psi_1 = \Psi_2 =\Phi$. Equivalently, whenever $\Phi \geq t \Psi$ for some $0< t < 1$ and $\Psi\in CP(\one)$, then $\Psi = \Phi$. 

For all $\Phi\in CP$, Arveson's Radon-Nikodym Theorem gives an explicit description of the  set $\{\Psi\in CP\ :\ \Phi -\Psi \geq 0\}$. Using this, he proved a characterization of the 
extreme points of $CP(\one)$. Later, Choi \cite{Choi75} gave a simplified treatment of Arveson's result in the case of matrix algebras. 
Every CP map on $M_N(\C)$ has a Kraus representation \cite{K71}:  There is a set $\{V_1,\dots,V_M\}\subset \cM_N(\C)$ such that for all $A \in \cM_N(\C)$,
\begin{equation}\label{kraus1}\Phi(A) = \sum_{j=1}^M V_j^* A V_j\ ,
\end{equation}
and one may always take such a representation in which $\{V_1,\dots,V_M\}$ is linearly independent; such a Kraus representation is {\em minimal}. There is a close relation between minimal 
Kraus representations and minimal Stinespring representations that is recalled in an appendix.
If \eqref{kraus1} is {\em any} Kraus representation of  $\Phi$, then $\Phi$  is unital if and only if $\sum_{j=1}^M V_j^*V_j = \one$, and Choi's matricial version \cite{Choi75} of 
Arveson's theorem on extremality is that if \eqref{kraus1} is a minimal Kraus representation, $\Phi$ is extremal in $CP(\one)$ if and only if $\{V_i^*V_j\ :\ 1\leq i,j \leq M\}$ is linearly independent. 

While necessary and sufficient conditions for a QMS generator to be self-adjoint for the GNS or KMS inner products have been proved by Alicki \cite{A78} and Fagnola an Umanita \cite{FU07},  
little appears to be known about the order structure and extreme points  for $CP_m$, $CP_m(\one)$ and $QMS_m$. Even for the case of the GNS or KMS inner products, there is more to say.

For the KMS inner product, let $CP_{KMS} := CP_{m_{1/2}}$. We prove the following results:

\begin{thm}\label{Main2A}  Let $\mathfrak{K}$ be the real vector space consisting of all $V\in \cM_N(\C)$ such that $\Delta^{-1/2} V = V^*$.  The extremal elements $\Phi$  of $CP_{KMS}$ are precisely the elements of the form 
\begin{equation}\label{KMSCPA}
\Phi(A) = V^*AV\ , \qquad V\in \mathfrak{K}\ .
\end{equation}
Every map in $CP_{KMS}$ is a positive \ linear combination of at most $N^2$ such maps.
\end{thm}

\begin{thm}\label{RNlemA}
Let $\Phi$ and $\Psi$ be two  CP maps that are KMS self-adjoint.  Let $\Phi(A) =  \sum_{j=1}^M V_j^*A V_j$ be a minimal Kraus representation of $\Phi$ with each 
$V_j\in \mathfrak{K}$. Then $\Phi - \Psi$ is CP if and only if there exists  a real  $M\times M$ matrix $T$  such that $0 \leq T \leq \one$ and
\begin{equation}\label{baskrid3A}
 \Psi(A) = \sum_{i,j=1}^M T_{i,j} V_i^* A V_j\ .
\end{equation}
\end{thm}

\begin{thm} Let $\Phi$ be a unital CP map that is  KMS self-adjoint and let $\Phi(A) = \sum_{j=1}^M V_j^*A V_j$ be a minimal Kraus representation of 
$\Phi$ with $V_j\in \mathfrak{K}$ for each $j$.  Then $\Phi$ is an extreme point of the set of unital CP maps that are KMS self-adjoint if and only if
\begin{equation}\label{lisetA}
\{ V_i^*V_j + V_j^*V_i\ :\ 1\leq i \leq j \leq M\}
\end{equation}
 is linearly independent over the real numbers.   
\end{thm}

Theorem~\ref{RNlemA} is an analog of Arveson's Radon-Nikodym  Theorem for CP maps in the context of KMS self-adjointness.  
The matrix algebra version of Arveson's theorem states that if  \eqref{kraus1} is a minimal Kraus representation of a CP map $\Phi$, and 
$\Psi$ is another CP map, then $\Phi - \Psi$ is CP if and only if there is a complex  $M\times M$ matrix $T$, $0 \leq T \leq \one$ such that $\Psi$ is given by \eqref{baskrid3A}.  
The restriction that both $\Phi$ and $\Psi$ are self-adjoint on $\H_{KMS}$ results in the further requirements that $V_j$ belongs to the real vector 
space $\mathfrak{K}$ for each $j$, and that $T$ is real.  The matricial version of Arveson's theorem in proved in an appendix.

For every $m\in \mathcal{P}[0,1]$, there is a natural order on $QMS_m$. While $QMS$ is not a pointed cone, $QMS_m$ is always a pointed cone:  We prove:
\begin{thm}\label{PointedA}  For any $m\in  \mathcal{P}[0,1]$, if  $\Phi\in QMS_m$ and $-\Phi\in QMS_m$, then $\Phi =0$.
\end{thm}

For $\cL_1,\cL_2\in QMS_m$, we define $\cL_1\geq \cL_2$ to mean that $\cL_1- \cL_2 \in QMS_m$. By Theorem~\ref{PointedA}, it follows that if $\cL_1\geq \cL_2$ and 
$\cL_2\geq \cL_1$, then $\cL_1 = \cL_2$. We then define an element $\cL\in QMS_m$ to be {\em extremal} in case whenever $\widetilde{\cL}\in QMS_m$ satisfies 
$\cL \geq t \widetilde \cL$ for some $t>0$, $\widetilde \cL$ is a non-zero multiple of $\cL$.  

The special feature of self-adjointness on $\H_m$ for $m = \delta_s$, and $s\in [0,1]$ that greatly simplifies the task of studying $QMS_m$ in these cases is that there is an orthogonal decomposition of 
$\widehat{\mathfrak{H}}$ into two subspaces, each of which is invariant under the operation of taking the adjoint on $\H_m$:

\begin{defi}\label{orthsdef} Define $\widehat{\mathfrak{H}}_{\mathcal{S}}$ to be the subspace of $\widehat{\mathfrak{H}}$ consisting of all operators $\Phi$ of the form
\begin{equation}\label{special1}
\Phi(A) = XA + AY
\end{equation}
for some $X,Y\in \cM_N(\C)$. 
\end{defi}

We shall prove:
\begin{lm}\label{decomplm}  For each $s\in [0,1]$, both $\widehat{\mathfrak{H}}_{\mathcal{S}}$ and $\widehat{\mathfrak{H}}_{\mathcal{S}}^\perp$ are invariant under the 
operation of taking the adjoint  with respect to the inner product  $\langle \cdot,\cdot \rangle_{\delta_s}$. 
Moreover, if $\{V_1,\dots,V_M\}$ is linearly independent in $\cM_N(\C)$, the map $A \mapsto \sum_{j=1}^M V_j^* A V_j$ belongs to $\widehat{\mathfrak{H}}_{\mathcal{S}}^\perp$ if and only if $\tr[V_j] = 0$ for each $j$. 
\end{lm}
(Of course, once the invariance of one subspace is shown, the invariance of the other follows).   

Now consider any QMS generator $\cL$. By the $LGKS$ Theorem \cite{Lin76, GKS76}  recalled earlier  in the introduction, $\cL$ has the form
\begin{equation}\label{LGKS}
\cL(A) = (G^*A + AG) +    \Psi(A)
\end{equation}
where $\Psi$ is CP.  Let 
$$
\Psi(A) = \sum_{j=1}^M V_j^* A V_j
$$ 
be a minimal Kraus representation of $\Psi$.   Replacing each $V_j$ by $V_j - \tr[V_j]\one$, and absorbing the difference into $G$, we may assume that $\tr[V_j] = 0$ for each $j$. 
By  Lemma~\ref{decomplm},  we then have $\Psi\in  \widehat{\mathfrak{H}}_{\mathcal{S}}^\perp$.
Furthermore, adding a purely imaginary multiple of $\one$ to $G$ does not change $\cL$, and hence we may also assume  without loss of generality that $\tr[G]\in \R$.     
 Thus, making these choices for $G$ and $\Psi$, \eqref{LGKS}
gives the decomposition of $\cL$ into its components in $\widehat{\mathfrak{H}}_{\mathcal{S}}$ and $\widehat{\mathfrak{H}}_{\mathcal{S}}^\perp$.  Then by Lemma~\ref{decomplm}, if 
$\cL$ is self-adjoint on $\H_{\delta_s}$, each of these pieces must be {\em individually} self-adjoint on $\H_{\delta_s}$.   For example, for $s=1/2$, corresponding to the KMS inner product, 
Theorem~\ref{Main2A} gives the necessary and sufficient conditions for $\Psi$ to be self-adjoint, and then an easy computation shows that
$A \mapsto G^*A + AG$ is self-adjoint on $\H_{KMS}$ if and only if $\Delta^{-1/2}G = G^*$, where we have taken, without loss of generality, $\tr[G]\in \R$. 

This brings us to a result of Fagnola and Umanita \cite[Theorem 37]{FU07}:  Let $\cL \in QMS$ be given in the form
$$
\cL(A) = G^*A + AG + \sum_{j=1}^M W_j^* A W_j
$$
where $\tr[G]\in \R$, $\{W_1,\dots,W_M\}$ is linearly independent, and for each $j$, $\tr[W_j]= 0$.   Then $\cL$ is self-adjoint on $\H_{KMS}$ is and only if 

\smallskip
\noindent{\it (i)} $\Delta^{-1/2}G = G^*$

\smallskip
\noindent{\it (ii)} There is an $M\times M$ unitary matrix $\widetilde U$ such that for each $j$,  $\Delta^{-1/2}W_j   = \sum_{k=1}^M \widetilde U_{j,k} W_k^*$. 
\smallskip

We have already explained how condition {\it (i)} follows from Lemma~\ref{decomplm} apart from the simple computation that will be provided below. As for {\it (ii)}, let $\Psi(A)= \sum_{j=1}^M W_j^* A W_j$. 
By Lemma~\ref{decomplm}, this must be self-adjoint on $\H_{KMS}$ and then by Theorem~\ref{Main2A} there is another minimal Kraus representation $\Psi(A)= \sum_{j=1}^M V_j^* A V_j$, necessarily with the same 
$M$, such that $\Delta^{-1/2}V_j=  V_j^*$.  By the unitary equivalence of minimal Kraus representations (see the appendix), there is an $M\times M$ unitary matrix such that for each $j$, $W_j = \sum_{k=1}^M U_{j,k}V_k$ and hence 
$$\Delta^{-1/2} W_j = \sum_{k=1}^M U_{j,k}\Delta^{-1/2}  V_k = \sum_{k=1}^M U_{j,k}  V_k^* = \sum_{k,\ell=1}^M U_{j,k}U_{\ell,k}  W_k^* = \sum_{\ell=1}^M (UU^T)_{j,\ell}  W_k^*$$
where $U^T$ denotes the transpose of $U$, and of course $\widetilde U := UU^T$  is unitary.  

Conversely, suppose that {\it (ii)} is satisfied.
Let $U$ be another unitary to be chosen below.  Then for each $\ell$,
$$\Delta^{-1/2} \left(\sum_{j=1}^M U_{\ell,j}W_j \right) = \sum_{j,k=1}^M U_{\ell,j} \widetilde U_{j,k}W_k^* \ ,$$
and thus if we choose $U$ so that $\overline{U} = U\widetilde U$, we may define $V_j  = \sum_{k=1}^M U_{j,k}W_k$, $\Psi(A) = \sum_{j=1}^M V_j^*A V_j$ and for each $j$, $\Delta^{-1/2} V_j = V_j^*$ and 
$\tr[V_j] =0$. It is easy to find such $U$: Choose an orthonormal basis in which $\widetilde{U}$ is diagonal with $j$th diagonal entry $e^{i\theta_j}$.  We may take $U$ to be diagonal in this same basis with $j$th diagonal entry $e^{-i\theta_j/2}$.

The result of Fagnola and Umanita allows one to check whether or not any given QMS generator is self-adjoint on $\H_{KMS}$. However, it does not  provide a parameterization of 
$QMS_{KMS}$, the cone of QMS generators that are self-adjoint on $\H_{KMS}$, nor can one readily read off the set of extreme points from their result; there are compatibility conditions relating $G$ and $\{W_1,\dots,W_M\}$.   The following result provides this and other additional information:

\begin{thm}\label{Main7A}  There is a one-to-one correspondence between elements $\cL$ of  $QMS_{KMS}$  and CP maps $ \Psi\in \widehat{\mathfrak{H}}_\mathcal{S}^\perp$ that are self-adjoint on $\H_{KMS}$.   
The correspondence identifies $\Psi$ with $\cL_\Psi$ where
\begin{equation}\label{gsw16}
\cL_\Psi(A) = G^*A + AG + \Psi(A)
\end{equation}
where $G = H+iK$, $H$ and $K$ self-adjoint and given by
\begin{equation}\label{HforKMS_intro}
H := \tfrac12\Psi(\one)
\end{equation}
and 
\begin{equation}\label{gsw17_intro}
K := \frac1i\int_0^\infty e^{-t   \sigma^{1/2}} ( \sigma^{1/2}  H - H  \sigma^{1/2})  e^{-t   \sigma^{1/2}}{\rm d} t\ .
\end{equation}
Furthermore,  for all $\cL_{\Psi_1},\cL_{\Psi_2} \in QMS_{KMS}$,  $\cL_{\Psi_1} - \cL_{\Psi_2} \in QMS$ if and only if $\Psi_1 - \Psi_2 \in CP$; i.e., $\cL_{\Psi_1} \geq \cL_{\Psi_2}$ if and only if 
$\Psi_1 \geq \Psi_2$, and the extreme points of $QMS_{KMS}$ are precisely the generators of the form
$$\cL(A) := G^*A + AG + V^*AV$$
where $\Delta^{-1/2}V = V^*$ and where $G=H+iK$ is given by \eqref{HforKMS_intro} and \eqref{gsw17_intro} for $\Psi(A) = V^*AV$. 
\end{thm} 

\begin{remark} It is easy to see that $K$ defined by \eqref{gsw17_intro} satisfies $\tr[K] =0$, so that $\tr[G]\in \R$.  
\end{remark}

We also prove analogous results for self-adjointness with respect to the inner products on $\H_{\delta_s}$ for all $s\neq 1/2$. This includes the GNS case, and in fact it is well known that for $s\neq 1/2$, any Hermitian operator 
$\Phi$ is self-adjoint on $\H_{\delta_s}$ for {\em all} $s\neq 1/2$ if and only if it is self-adjoint  on $\H_{GNS} = \H_{\delta_0}$. It follows that Hermitian operators $\Phi$ that are self-adjoint on 
$\H_{GNS}$ are {\em universally self-adjoint} --  they are self-adjoint  on $\H_m$ for all $m$. The results specifying extreme points, e.g., of the set of GNS self-adjoint unital CP maps, etc., are new, 
as is the theorem giving necessary conditions for $\Phi_1 - \Phi_2 \in CP$ when $\Phi_1$ and $\Phi_2$ are CP and self-adjoint on $\H_{GNS}$, while the structue of $QMS_{GNS}$ was worked out by Alicki \cite{A78}.

We turn next to the BKM inner product.  It is much more difficult to prove analogs of the theorems proved here for the KMS and GNS inner product, in part because 
$\widetilde{\mathfrak{H}}_{\mathcal{S}}^\perp$ is not invariant under the operation of taking the BKM adjoint. However, the BKM case shares one very nice feature with the KMS case;
in both cases $\cM_m^{-1}$ is CP.  Using this, we construct and study a one-to-one map from $CP_{KMS}$ into $CP_{BKM}$. This map is unfortunately not surjective, but it does also take unital 
maps to unital maps, and so it gives us a large class of quantum operations that are self-adjoint on $\H_{BKM}$.   

Finally, we study operators that are {\em evenly self-adjoint}, by which we mean self-adjoint on $\H_{m}$ for all even $m$.  The class of evenly self-adjoint QMS generators is strictly larger than the class of 
GNS self-adjoint QMS generators, but also strictly smaller  than the set of QMS generators that are BKM or KMS self-adjoint.

\section{Background material}

We recall some useful tools.
It is convenient to index orthonormal bases of $\fH$ by ordered pairs $(i,j) \in \{1,\dots N\}\times  \{1,\dots N\} =: \cJ_N$. We use lower case greek letters to denote elements of the index set. For $\alpha = (i,j)\in \cJ_N$, $\alpha':= (j,i)$.

For $F,G\in \cM_N(\C)$, define the operator $\#(F\otimes G)$ on $\fH$ by
\begin{equation}\label{shadef}
\#(F\otimes G)X := FXG\ .
\end{equation}
Simple computations show that 
\begin{equation}\label{orth}
\langle \#(F_1\otimes G_1), \#(F_2\otimes G_2)\rangle_{\widehat{\fH}} = \langle F_1,F_2\rangle_{\fH} \langle G_1,G_2\rangle_{\fH}\ .
\end{equation}
Hence, if $\{F_\alpha\}_{\alpha\in \cJ_N}$   and $\{G_\alpha\}_{\alpha\in \cJ_N}$  are two orthonormal bases of $\fH$,  $\{\#(F_\alpha\otimes G_\beta)\}_{\alpha,\beta\in \cJ_N}$ is an orthonormal basis of $\widehat{\fH}$. 

Now fix any orthonormal basis $\{F_\alpha\}_{\alpha\in \cJ_N}$ of $\fH$.
Then $\{F^*_\alpha\}_{\alpha\in \cJ_N}$  is also an orthonormal basis of $\fH$, and hence $\{\#(F_\alpha^*\otimes F_\beta)\}_{\alpha,\beta\in \cJ_N}$ is an orthonormal basis of $\widehat{\fH}$.   Thus,
 every linear operator $\Phi$ on $\fH$ has an expansion
\begin{equation}\label{GKS1}
\Phi = \sum_{\alpha,\beta\in \cJ_N} (c_\Phi)_{\alpha,\beta}\#(F_\alpha^*\otimes F_\beta)\ ,
\end{equation}
where
\begin{equation}\label{GKS2}
(c_\Phi)_{\alpha,\beta} = \langle \#(F_\alpha^*\otimes F_\beta), \Phi \rangle_{\widehat{\fH}}\ .
\end{equation}
In particular, the coefficients $(c_\Phi)_{\alpha,\beta}$ are uniquely determined. The following definition is from \cite{GKS76}.

\begin{defi}[Characteristic matrix] Given a linear operator $\Phi$ on $\fH$, and an orthonormal basis $\{F_\alpha\}_{\alpha\in \cJ_N}$ of $\fH$, 
its {\em characteristic matrix} for this orthonormal basis is the $N^2\times N^2$ matrix  $C_\Phi$ whose $(\alpha,\beta)$th entry is $(c_\Phi)_{\alpha,\beta}$ as specified in 
\eqref{GKS2}. 
\end{defi}

\begin{remark}\label{exg} An easy computation shows the following: Let $\Phi$ be a linear operator on $\fH$.  Let $\{F_\alpha\}$ and $\{\widetilde{F}_\alpha\}$ be two orthonormal bases of $\fH$.   Let $C_\Phi$ be the characteristic matrix of $\Phi$ with respect to $\{F_\alpha\}$,
and let $\widetilde{C}_\Phi$ be the characteristic matrix of $\Phi$ with respect to $\{\widetilde{F}_\alpha\}$.  Let $U$ be the $N^2\times N^2$  unitary matrix such that
${\displaystyle 
\widetilde{F}_\alpha = \sum_{\beta} U_{\alpha,\beta} F_\beta}$. 
Then
${\displaystyle
\tilde C_\Phi = UC_\Phi U^*}$.
\end{remark}

\begin{lm}[See \cite{GKS76}]\label{Hchar}  A linear operator $\Phi$ on $\fH$ is Hermitian if and only if $C_\Phi$, for any orthonormal basis $\{F_\alpha\}_{\alpha\in \cJ_N}$, is self-adjoint. 
\end{lm}

\begin{proof}  We compute
$$(\Phi(A^*))^* := \left( \sum_{\alpha,\beta\in \cJ_N} (c_\Phi)_{\alpha,\beta}F_\alpha^* A^* F_\beta \right)^* = \sum_{\alpha,\beta\in \cJ_N} \overline{(c_\Phi)_{\alpha,\beta}}F_\beta^* A^* F_\alpha =
\sum_{\alpha,\beta\in \cJ_N} \overline{(c_\Phi)_{\beta,\alpha}}F_\alpha^* A^* F_\beta \ . 
$$
By the uniqueness of the coefficients, $\Phi$ is Hermitian if and only if $(c_\Phi)_{\alpha,\beta} = \overline{(c_\Phi)_{\beta,\alpha}}$ for all $\alpha,\beta$. 
\end{proof}

The following lemma is a variant of Choi's Theorem \cite{Choi75}; see \cite{GKS76} for this version.

\begin{lm}\label{CPchar}  A linear operator $\Phi$ on $\cM_N(\C)$  is completely positive if and only if for every orthonormal basis $\{F_\alpha\}_{\alpha\in \cJ_N}$ of $\fH$, the corresponding characteristic matrix $C_\Phi$ is positive semi-definite. 
\end{lm}

\begin{proof} 
The right side of \eqref{GKS2} can be computed using the matrix unit basis to compute the trace:
\begin{equation*}
(c_\Phi)_{\alpha,\beta} = \langle \#(F_\alpha^*\otimes F_\beta ), \Phi \rangle_{\widehat{\fH}} = \frac{1}{N^2} \sum_{1 \leq k,\ell\leq N} \tr[E_{k,\ell}^* F_\alpha\Phi(E_{k,\ell})F_\beta^*]\\
\end{equation*} 
Let $(z_1,\dots,z_{N^2})\in \C^{N^2}$ and define $G := \sum_{\alpha\in \cJ_N} z_\alpha F_\alpha^*$. Then
\begin{equation}\label{GKSA1}
\sum_{\alpha,\beta }\overline{z}_\alpha (c_\Phi)_{\alpha,\beta} z_\beta   =  \frac{1}{N^2} \sum_{1 \leq k,\ell\leq N} \tr[E_{\ell,k} G^*\Phi(E_{k,\ell})G]\ .
\end{equation}
Let $[E_{i,j}]$ denote the block matrix whose $i,j$th entry is $E_{i,j}$. Then it is easy to see that $N^{-1/2}[E_{i,j}]$ is an orthogonal projection, and in particular, positive. 
Now suppose that $\Phi$ is completely positive. Then the block matrix $[G^*\Phi(E_{i,j})G]$ whose $i,j$th entry is $G^*\Phi(E_{i,j})G$ is positive. The right side of \eqref{GKSA1} is then
the trace (on the direct sum of $N$ copies of $\C^N$) of the product of positive $N^2\times N^2$ matrices, and as such it is positive.  Thus, whenever $\Phi$ is completely positive, $C_\Phi$ is positive semi-definite.

On the other hand, suppose that $C_\Phi$ is positive semi-definite. Let $\Lambda$ be a diagonal matrix whose diagonal entries are the eigenvalues of $C_\Phi$, and let $U$ be a unitary such that
$C_\Phi = U^*\Lambda U$.
Then by \eqref{GKS1}, for any $X\in \fH$,
$$
\Phi(X) =   \sum_{\alpha,\beta,\gamma\in \cJ_N}  U^*_{\alpha,\gamma} \lambda_k U_{\gamma,\beta }F_\alpha^*XF_\beta  =  \sum_{\gamma\in \cJ_N}  V_\gamma^* X V_\gamma \ ,
$$
where ${\displaystyle V_\gamma := \sqrt{\lambda_\gamma}\sum_{\alpha\in \cJ_N} U_{\gamma,\alpha}F_\alpha}$.
This shows that whenever $C_\Phi$ is positive semi-definite, $\Phi$ is completely positive, and provides a Kraus representation of it. 
\end{proof}

So far, we have not required any special properties of the orthonormal bases $\{F_\alpha\}_{\alpha\in \cJ_N}$ of $\fH$ that we used. Going forward, it will be necessary to choose bases that have several useful properties:

\begin{defi}[symmetric, unital and matrix unit  bases]\label{unbadef}  
An orthonormal basis $\{F_\alpha\}_{\alpha\in \cJ_N}$ of $\fH$ is  {\em symmetric} in case
\begin{equation}\label{symmetric}
{\rm for \ all}\quad \alpha\in \cJ_N\ ,\quad F_\alpha^* = F_{\alpha'}\ .
\end{equation}
It is {\em unital}  in case it is symmetric and moreover
\begin{equation}\label{unital}
F_{(1,1)} = \one\ .
\end{equation}
It is the {\em matrix unit} basis corresponding to an orthonormal basis $\{u_1,\dots,u_N\}$ of $\C^N$ in case
\begin{equation}\label{matrixunit}
F_{(i,j)} =\sqrt{N}|u_i\rangle \langle u_j|\ .
\end{equation}
Note that a matrix unit basis is symmetric. 
\end{defi}

One reason unital bases are useful is the following: the characteristic matrix $C_I$ of the identity transformation $I(A) = A$ in a unital basis $\{F_\alpha\}_{\alpha\in \cJ_N}$ has only one nonzero entry, a 1 in the upper left corner. Indeed, for any $A$, the expansion of $I(A)$ in the form $\sum_{\alpha\beta} \#(F_\alpha^*\otimes F_\beta)(A) = \sum_{\alpha\beta} (c_I)_{\alpha,\beta} F_\alpha^\ast A F_\beta$ reads simply $\one^\ast A \one = F_{(1,1)}^\ast A F_{(1,1)}$. This fact will be used later.

\begin{remark}\label{remark_unital_basis}
When studying self-adjointness with respect to the various inner products we have introduced, the following bases will be particularly useful:
Let $\sigma\in \Dens_+(\aA)$, and let $\{u_1,\dots,u_N\}$  be an orthonormal basis  of $\C^N$ consisting of eigenvectors of $\sigma$:
$$
\sigma u_j = \lambda_j u_j \ , \quad j=1,\dots,N\ .
$$
The associated matrix unit basis is then given by 
\begin{equation}\label{matrixunits}
E_{(i,j)} =\sqrt{N}|u_i\rangle \langle u_j|\ .
\end{equation}
We can then construct a unital basis from this as follows:
Let $v_1$ be the unit vector $N^{-1/2}(1,\dots,1) \in \C^N$  each of whose entries $(v_1)_k = N^{-1/2}$.  Let $e_1 = (1,0,\dots,0)$ be the unit vector in $\C^N$ whose first entry is $1$.  Define
$$u := \frac{1}{\|v_1 - e_1\|}(v_1 - e_1) \qquad{\rm and}\qquad  V := \one - 2 |u \rangle \langle u|\ .$$
Then it is evident that $V$ is a self-adjoint real unitary matrix, and that $V e_1 = v_1$. ($V$ is the Householder reflection of $e_1$ onto $v_1$). Since $Ve_1$ is the first column of $V$, the first column of $V$ is $v_1$, and then since $V$ is symmetric, the first row of $V$ is also $v_1$. One readily finds that
\begin{equation}\label{Vdef}
\sqrt{N} V_{i,j} = \begin{cases} \phantom{xx} 1 & i=1 \ {\rm or}\ j=1\\
\frac{N-2}{\sqrt{N}-1} -1 & i=j \geq 2\\
\frac{-1}{\sqrt{N}-1} & {\rm otherwise}\end{cases}
\end{equation}
 Let $U$ be the $N^2\times N^2$ unitary matrix  whose lower right $(N^2-N)\times(N^2-N)$ block is the identity, and whose upper left $N \times N$ block is $V$.  That is, for $N=2$,
  \begin{equation}\label{passtounit2}
U = \left[ \begin{array}{cccc} \frac{1}{\sqrt{2}} & \phantom{-}\frac{1}{\sqrt{2}} & 0 & 0\\
 \frac{1}{\sqrt{2}} & -\frac{1}{\sqrt{2}} & 0 & 0\\
 0 & \phantom{-}0 & 1 & 0\\
 0 & \phantom{-}0 & 0 & 1\end{array}\right]\ .
 \end{equation}
Then using this unitary $U$, we define
\begin{equation}\label{unitals}F_\alpha := \sum_{\beta}U_{\alpha,\beta} E_\beta\ .
\end{equation}
Then  $\{F_\alpha\}_{\alpha\in \cJ_N}$ is a unital basis; in particular $F_{(1,1)} = \one$. 
\end{remark}

We are now ready to recall the characterization of $QMS$ generators.  Again, the definition is due to \cite{GKS76}.

\begin{defi}[Reduced characteristic matrix]  Let $\cL$ be a Hermitian operator on $\fH$ such that $\cL\one = 0$, and let $\{F_\alpha\}_{\alpha\in \cJ_N}$ be a unital orthonormal basis. 
Let $C_\cL$ be the characteristic matrix of $\cL$ with respect to this basis.  The {\em reduced characteristic matrix} $R_\cL$ of $\cL$ is the $(N^2-1)\times (N^2-1)$ matrix obtained by deleting the first row and column of $C_{\cL}$.
\end{defi}

The point of this definition is the following, due to \cite{GKS76,PZ77}.

\begin{lm}\label{redLm} Let $\cL$ be a Hermitian operator on $\fH$ such that $\cL\one = 0$, and let $\{F_\alpha\}_{\alpha\in \cJ_N}$ be a unital orthonormal basis. Let $R_\cL$ be the  reduced characteristic matrix of $\cL$ with respect to this basis.  Then $\cL$ is a QMS generator if and only if  $R_\cL$ is positive semi-definite.
\end{lm}

\begin{proof} Suppose that $\cP_t := e^{t\cL}$ is completely positive for each  $t>0$. By \eqref{chid}, $R_I$, the reduced characteristic matrix of the identity, is $0$.
Then since
$$R_{t^{-1}(\cP_t - I)} = t^{-1}R_{\cP_t} -  t^{-1}R_I  =  t^{-1}R_{\cP_t}  \ ,$$
the reduced characteristic matrix of $t^{-1}(\cP_t - I)$ coincides with the reduced characteristic  matrix of 
$t^{-1}\cP_t $, and by Lemma~\ref{CPchar} this is positive. Taking the limit $t\to 0$, we conclude that the reduced characteristic matrix of $\cL$ is positive.

Conversely, suppose that the reduced characteristic matrix of $\cL$ is positive. 
Since $F_{(1,1)} = \one$, 
\begin{eqnarray*}
\cL(A)  = \sum_{\alpha,\beta}(c_\cL)_{\alpha,\beta} F_\alpha^* A F_\beta = G^*A + AG  +   \sum_{\alpha,\beta}(r_\cL)_{\alpha,\beta} F_\alpha^* A F_\beta
\end{eqnarray*}
where 
$$G := \frac12 (c_\cL)_{(1,1),(1,1)} \one + \sum_{\beta} (c_\cL)_{(1,1),\beta} F_\beta \ .
$$
By Lemma~\ref{CPchar}, if we define $\Psi(A) :=  \sum_{\alpha,\beta}(r_\cL)_{\alpha,\beta} F_\alpha^* A F_\beta$, then 
$\Psi$  is completely positive. Defining $\Phi(A) := G^*A + AG$, we have $\cL = \Phi+\Psi$, and then by the  argument around \eqref{trotter}, $e^{t\cL}$ is completely positive for all $t> 0$. Finally, since $\cL\one =0$, $e^{t\cL}$ is a QMS. 
\end{proof}

We are also interested in characterizing the characteristic matrices of operators $\Phi$ in $\fH$  for which $\Phi\one =0$, and, for some $\sigma\in \Dens_+$, $\Phi^\dagger \sigma = 0$.   
 \begin{lm}\label{pairlem}
 Let $\{E_\alpha\}$ be a matrix unit basis of $\fH$. Let $\Phi$ be an operator on $\fH$, and let $C_\Phi$ be its characteristic matrix with respect to $\{E_\alpha\}$.  Then $\Phi(\one) = 0$ if and only if  for each $1\leq k,\ell\leq N$
 \begin{equation}\label{anhilone}
 \sum_{j=1}^N (c_\Phi)_{(j,k),(j,\ell)} = 0\ ,
 \end{equation}
 and $\Phi(\one) = \one$ if and only if \eqref{anhilone} is satisfied for $k\neq \ell$, and for $1\leq k \leq N$,
 \begin{equation}\label{oneone}
 \sum_{j=1}^N (c_\Phi)_{(j,k),(j,k)} = \frac1N\ .
 \end{equation}
 
 \end{lm}
 \begin{proof} We compute
 $$\Phi(\one)  = \sum_{\alpha,\beta} (c_\Phi)_{\alpha,\beta}E_\alpha^* \one E_\beta = \sum_{\alpha,\beta} (c_\Phi)_{\alpha,\beta}\delta_{\alpha_1,\beta_1}E_{\alpha_2,\beta_2} = \sum_{j,k,\ell} (c_\Phi)_{(j,k),(j,\ell)} E_{(k,\ell)} \ , $$
 from which \eqref{anhilone} follows, and then \eqref{oneone} follows since $\one =\frac1N\sum_{k=1}^N E_{(k,k)}$.
 \end{proof}
 
\begin{lm}\label{daglem}
Let $\Phi$ be an Hermitian operator on $\fH$, and let $C_\Phi$  be its characteristic matrix with respect to a symmetric orthonormal basis $\{F_\alpha\}_{\alpha\in \cJ_N}$ of $\fH$.  Then 
\begin{equation}\label{daggform}
( \#(F_\alpha^* \otimes F_\beta))^\dagger = \#F_{\alpha'}^* \otimes F_{\beta'}\ ,
\end{equation}
and
\begin{equation}\label{daggform2}
(c_{\Phi^\dagger})_{\alpha,\beta}  =  (c_\Phi)_{\beta',\alpha'}\ .
\end{equation}
\end{lm}

 \begin{proof} We compute
$$\tr[B^* \#(F_\alpha^* \otimes F_\beta)(A)] = \tr[ B^* F_\alpha^*AF_\beta] =  \tr[ (F_\alpha B F_\beta^*)^*A] = \tr[ (F^*_{\alpha'} B F_{\beta'})^*A] \ , $$
and this proves \eqref{daggform}.  
 Then since 
${\displaystyle \Phi = \sum_{\alpha,\beta}(c_\Phi)_{\alpha,\beta} \#(F_\alpha^* \otimes F_\beta)}$ and since 
 $\overline{(c_\Phi)_{\alpha,\beta} } = (c_\Phi)_{\beta,\alpha}$ by Lemma~\ref{Hchar},
$$
\Phi^\dagger = \sum_{\beta,\alpha}\overline{(c_\Phi)_{\alpha,\beta}}( \#(F_\alpha^* \otimes F_\beta))^\dagger \ .
$$
Then by \eqref{daggform},  \eqref{daggform2} follows. 
\end{proof}

The next lemma says that for $\Phi \in CP(\one)$ or $\cL \in QMS$, there is often exactly one $\sigma\in \Dens_+$ with respect to which $\Phi$ or $\cL$ can possibly  be self-adjoint on $\H_m$ for that choice of $\sigma$.  

\begin{lm}\label{uniqlm}  Let $\Phi$ be a Hermitian operator that is self-adjoint on $\H_m$ where $\H_m$ is defined in terms of some $\sigma\in \Dens_+$. If $\Phi(\one) = \one$, then $\Phi^\dagger(\sigma) = \sigma$, and if 
$\Phi(\one) = 0$, then $\Phi^\dagger(\sigma) =0$.  Hence if $\Phi\in CP(\one)$ is such that $1$ is a non-degenerate eigenvalue of $\Phi$, then for any $m\in \mathcal{P}[0,1]$, there can be only one $\sigma\in \Dens_+$ such that $\Phi$ is self-adjoint on $\H_m$.
Likewise, if $\cL\in QMS$ is such that $0$ is a non-degenerate eigenvalue of $\cL$, then for any $m\in\mathcal{P}[0,1]$, there can be only one $\sigma\in \Dens_+$ such that $\cL$ is self-adjoint on $\H_m$.

\end{lm}

\begin{proof}
For any  Hermitian operator $\Phi$ in $\fH$,  any  $\sigma\in \Dens_+$ and  any $m\in \mathcal{P}[0,1]$ such that  $\Phi$ is self-adjoint on $\H_m$, we compute that for all $A$,
  $$\tr[\Phi^\dagger(\sigma)A] = \tr[\sigma\Phi(A)] = \langle \one, \Phi(A)\rangle_m = \langle \Phi(\one), A\rangle_m\ .$$
  If $\Phi(\one) = \one$, we have $\tr[\Phi^\dagger(\sigma)A] = \tr[\sigma A]$ and since $A$ is arbitrary $\Phi^\dagger(\sigma) = \sigma$. Likewise, if $\Phi(\one) = 0$, we have $\tr[\Phi^\dagger(\sigma)A] = 0$ and since $A$ is arbitrary $\Phi^\dagger(\sigma) = 0$.
  \end{proof}
  
On account of this lemma, we will only rarely make the choice of $\sigma$ explicit in our notation.  In the next section we determine necessary and sufficient conditions for self-adjointness.

\section{characterization for self-adjointness with respect to $\langle \cdot, \cdot \rangle_m$}

 For   $m\in\mathcal{P}[0,1]$, let  $(\cdot,\cdot)_m$  denote the {\em $m$-weighted mean} of two nonnegative numbers $x,y$:
\begin{equation}
    (x,y)_m = \int_0^1 x^{s}y^{1-s} \; \mathrm{d}m(s) \quad , \quad x,y\in\mathbb{R}_+ \ .
\end{equation}
Note that for all $m\in \mathcal{P}[0,1]$,
\begin{equation}
 \min\{x,y\} \leq (x,y)_m \leq \max\{x,y\} \qquad{\rm and}\qquad    (x,x)_m = x \quad \text{for all } x\in\mathbb{R}_+ \ .
    \label{avg_x_x}
\end{equation}
For $m$ even,
\begin{equation}
    (x,y)_m = (y,x)_m \quad \text{for all } x,y\in\mathbb{R}_+\ .
\end{equation}
Otherwise define $\hat{m}$ to be the reflection of $m$ about $1/2$; i.e., $ \hat{m}(U) = m(1-U)$. Then 
\begin{equation}\label{ref}
    (x,y)_m = (y,x)_{\hat{m}} \quad \text{for all } x,y\in\mathbb{R}_+\ .
\end{equation}

\begin{remark}
\label{KMSmu}
For any even $m\in\mathcal{P}[0,1]$ and any $x,y\in\mathbb{R}_+$, $x\neq y$, we have
\begin{equation}
    \sqrt{xy} = (x,y)_{\text{KMS}} \leq (x,y)_m \ ,
\end{equation}
with equality if and only if $m = \delta_{1/2}$. Indeed, first use symmetry of $m$ to rewrite
\begin{equation}
    (x,y)_m = \frac{1}{2}\int_0^1 ( x^sy^{1-s} + x^{1-s}y^s ) \mathrm{d}m(s) \ .
    \label{integral_average}
\end{equation}
Then 
$$  ( x^sy^{1-s} + x^{1-s}y^s ) -\sqrt{xy} = \frac12(x^{s/2}y^{(1-s)/2} + x^{(1-s)/2}y^{s/2})^2\ ,$$
and integrating against $m$,
$$(x,y)_m - (x,y)_{\text{KMS}} = \frac12\int_0^1(x^{s/2}y^{(1-s)/2} - x^{(1-s)/2}y^{s/2})^2{\rm d}m(s)\ ,$$
and for $x\neq y$, the integrand vanishes only for $s =\frac12$. 
\end{remark}

For   $m\in\mathcal{P}[0,1]$ let $\mathcal{M}_m$ be the linear transformation on $\fH$ defined in \eqref{integral_M}. Let $\{u_1,\dots,u_N\}$ be a complete orthonormal basis of $\C^N$ consisting of eigenvectors of $\sigma$ so that $\sigma u_j = \lambda_ju_j$ for each $j$. By hypothesis, each $\lambda_j$ is strictly positive. Let $\{E_\alpha\}_{\alpha\in \cJ_N}$ be the associated matrix unit basis;
\begin{equation}\label{onbas}
E_{(j,k)} := \sqrt{N} |u_j\rangle\langle u_k|\ .
\end{equation}
Then
$$\mathcal{M}_m(E_{(j,k)} ) = (\lambda_j,\lambda_k)_m E_{(j,k)}\ .$$
That is, $\mathcal{M}_m$ is diagonalized by the orthonormal basis of $\fH$ consisting of $\{\sqrt{N}E_{(j,k)}\}_{1\leq j,k\leq N}$. Since all of the eigenvalues are strictly positive, it follows that $\mathcal{M}$ is a positive, invertible operator on $\fH$, and the inverse is given by
$$\mathcal{M}_m^{-1}(E_{(j,k)} ) = \frac{1}{(\lambda_j,\lambda_k)_m} E_{(j,k)}\ .$$

Since $E_{(j,k)}^* = E_{(k,j)}$, we also have, with $\widehat{m}$ denoting the reflection of $m$ about $s=1/2$, 
\begin{equation}\label{hermM}
\mathcal{M}_m(A^* ) = (\mathcal{M}_{\hat{m}}(A ))^*  \qquad{\rm and}\qquad  \mathcal{M}_m^{-1}(A^* ) = (\mathcal{M}^{-1}_{\hat{m}}(A ))^*
\end{equation}
In particular, if $m$ is even, $\mathcal{M}_m$ and $\mathcal{M}^{-1}_m$ are Hermitian. 

\begin{lm}\label{charMM} With respect to a matrix unit basis $\{E_\alpha\}_{\alpha\in \cJ_N}$ associated to $\sigma$, the characteristic matrices $C_{\mathcal{M}_m}$ and $C_{\mathcal{M}^{-1}_m}$ of $\mathcal{M}_m$
and $\mathcal{M}_m^{-1}$ are given by
\begin{equation}\label{charMM}
(c_{\mathcal{M}_m})_{\alpha,\beta} = (\alpha_1,\beta_1)_m\delta_{\alpha_1,\alpha_2}\delta_{\beta_1,\beta_2} \quad{\rm and}\quad 
(c_{\mathcal{M}^{-1}_m})_{\alpha,\beta} = (\alpha_1,\beta_1)^{-1}_m\delta_{\alpha_1,\alpha_2}\delta_{\beta_1,\beta_2}
\end{equation}\ .
\end{lm}

\begin{proof}  We compute
\begin{eqnarray*}(c_{\mathcal{M}_m})_{\alpha,\beta} &=& \frac{1}{N^2} \sum_{i,j}\tr[(E^*_\alpha E_{(i,j)} E_\beta))^*\cM_m(E_{(i,j)})] =
\frac{1}{N^2} \sum_{i,j}(\lambda_i,\lambda_j)_m \tr[(E^*_\alpha E_{(i,j)} E_\beta))^*E_{(i,j)}]\\
&=& \frac{1}{N^2} \sum_{i,j}(\lambda_i,\lambda_j)_m \tr[E^*_\beta E_{(j,i)} E_\alpha E_{(i,j)}] = \sum_{i,j}(\lambda_i,\lambda_j)_m \delta_{\beta_1,j}\delta_{\beta_2,j} \delta_{\alpha_1,i}\delta_{\alpha_2,i}\ ,
\end{eqnarray*}
form which the first formula follows. The second follows in the same way. 
\end{proof}

Hence, if we order the basis $\{E_\alpha\}_{\alpha\in \cJ_N}$ so that the first $N$ unit vectors are
$$E_{(1,1)}, \dots, E_{(N,N)}$$
in this order, then $C_{\mathcal{M}_m}$ and $C_{\mathcal{M}^{-1}_m}$ are both zero except in their upper left $N\times N$ blocks where the $(i,j)$th entries are $(\lambda_i,\lambda_j)_m$ and 
$(\lambda_i,\lambda_j)_m^{-1}$ respectively.  

Now let $A$ be a symmetric $N\times N$ matrix with positive entries.  Then a necessary  condition for $A$ to be positive semi-definite is that for all $i,j$ $A_{i,j} \leq \sqrt{A_{i,i} A_{j,j}}$. 
  For $N=2$, the condition is also sufficient, but it is easy to see that sufficiency fails already for $N=3$. 

However, combining the necessity with Remark~\ref{KMSmu}, we see that $\cM_m$ is never completely positive except in the KMS case, $m = \delta_{1/2}$, corresponding to the geometric mean. On the other hand, at least in the $2\times2$ case, $\cM_m^{-1}$ is completely positive whenever $m$ is even. There are  important cases in which $\cM_m^{-1}$ is completely positive in every dimension, including the KMS and BKM cases. We return to this later.

\begin{defi} Let $\Phi$ be a linear transformation on $\cM_N(\C)$, and hence on $\H_m$ for each  $m\in\mathcal{P}[0,1]$.  Its adjoint with respect to the inner product $\langle \cdot,\cdot\rangle_m$ is denoted $\Phi^{*,m}$.
\end{defi}

\begin{lm}\label{adlem}  Let $\Phi$ be an Hermitian  linear transformation on $\cM_N(\C)$ and let $m\in\mathcal{P}[0,1]$.  Then
\begin{equation}\label{adform}
\Phi^{*,m}=  \mathcal{M}^{-1}_{{m}}\circ  \Phi^\dagger \circ   \mathcal{M}_{{m}}\ .
\end{equation}
\end{lm}

\begin{proof} First recall that, since $\Phi$ is Hermitian, $\Phi^\dagger$ is Hermitian. To see this, note that
\begin{eqnarray*}
\tr[(\Phi^\dagger(B^*))^*A] &=& \tr[B \Phi(A)] = \overline{\tr[B^*(\Phi(A))^*]} =  \overline{\tr[B^*\Phi(A^*)]} \\
&=&\overline{\tr[(\Phi^\dagger B)^*A^*]}  = \tr[(\Phi^\dagger B)A]\ .
\end{eqnarray*}
Thus, $(\Phi^\dagger(B^*))^* = \Phi^\dagger(B)$ for all $B$. 

We now compute
\begin{eqnarray*}
\langle B, \Phi(A)\rangle_m &=& \tr[B^* \mathcal{M}_m(\Phi(A))] =  \tr[(\mathcal{M}_m(B))^*\Phi(A)] \\
&=&  \tr[(\Phi^\dagger(\mathcal{M}_m(B)))^* A]  \\
&=&    \tr[(\mathcal{M}_m^{-1}(\Phi^\dagger(\mathcal{M}_m(B))))^* \mathcal{M}_m(A)] 
\end{eqnarray*}
Then \eqref{adform} follows from \eqref{hermM}.  
\end{proof}

It follows that an Hermitian  operator $\Phi$ is self-adjoint on $\H_m$ if and only if
\begin{equation}\label{sacahr}
 \mathcal{M}_{{m}}\circ \Phi = \Phi^\dagger \circ  \mathcal{M}_{{m}}\ .
\end{equation}

\begin{lm}\label{adlem2} Let $\{E_\alpha\}$ be the orthonormal basis of $\fH$ specified in \eqref{onbas}. Then for all $m\in \mathcal{P}[0,1]$,
\begin{equation}\label{mb1}
\#(E_\alpha^* \otimes E_\beta )\circ \cM_m = (\lambda_{\alpha_1},\lambda_{\beta_1})_m \#(E_\alpha^* \otimes E_\beta )
\end{equation}
and 
\begin{equation}\label{mb2}
 \cM_m \circ \#(E_\alpha^* \otimes E_\beta )  = (\lambda_{\alpha_2},\lambda_{\beta_2})_m \#(E_\alpha^* \otimes E_\beta )
\end{equation}
\end{lm} 

\begin{proof}  We compute 
$$E_\alpha^* \cM_m(A) E_\beta = \int_0^1  E_\alpha^* \sigma^{1-s}A \sigma^s  E_\beta {\rm d}m = (\lambda_{\alpha_1},\lambda_{\beta_1})_m E_\alpha^* A  E_\beta\ ,$$
and this proves \eqref{mb1}.

Next, since $\cM^\dagger_m = \cM_m$, by \eqref{daggform}
$$
\cM_m \circ \#(E_\alpha^* \otimes E_\beta )  = (\#(E_{\alpha'}^* \otimes E_{\beta'} )\circ \cM_m)^\dagger =  (\lambda_{\alpha_2},\lambda_{\beta_2})_m (\#(E_{\alpha'}^* \otimes E_{\beta'} ))^\dagger
$$
form which \eqref{mb2} follows upon another application of \eqref{daggform}. 
\end{proof}

\begin{thm}\label{main1}  Let $\Phi$ be an Hermitian operator on $\fH$. Let  $\{E_\alpha\}$ be the orthonormal basis of $\fH$ specified in \eqref{onbas}, and let $C_\Phi$ be the characteristic matrix of $\Phi$ with respect to this basis. Then 
$\Phi$ is self-adjoint on $\H_m$ if and only if  for all $\alpha,\beta$, 
\begin{equation}\label{mb3}
 (c_\Phi)_{\alpha,\beta} = \frac{(\lambda_{\alpha_1},\lambda_{\beta_1})_m}{(\lambda_{\alpha_2},\lambda_{\beta_2})_m}  (c_\Phi)_{\beta',\alpha'}
\end{equation}
In particular, define the matrix $B_{\Phi,m}$ by
\begin{equation}\label{mb4}
(b_{\Phi,m})_{\alpha,\beta}  =  (c_\Phi)_{\alpha,\beta} (\lambda_{\alpha_2},\lambda_{\beta_2})_m\ ,
\end{equation}
and the anti-unitary self-adjoint operator $U$ on $\C^{N^2}$ given by
\begin{equation}\label{mb6}
(Uv)_\alpha = \overline{v_{\alpha'}}\  .
\end{equation}
 Then \eqref{mb5} is equivalent to
\begin{equation}\label{mb5}
UB_{\Phi,m} = B_{\Phi,m}U \ .
\end{equation}
\end{thm}

\begin{proof} Applying Lemma~\ref{adlem} together with Lemma~\ref{adlem2} yields \eqref{mb3}.  Then \eqref{mb3} can be written as 
\begin{equation}\label{mb7}
(b_{\Phi,m})_{\alpha,\beta}  =  \overline{(b_{\Phi,m})}_{\alpha',\beta'} \ ,
\end{equation}
Then \eqref{mb5} follows from \eqref{mb6} and the definition of $U$.
\end{proof}

\begin{remark} We may identify $\C^{N^2}$ equipped with its usual inner product  and $\cM_N(\C)$ equipped with the Hilbert-Schmidt inner product  in the usual way, identifying the vector $v$ that has entries $v_\alpha$ with the matrix $V$ that has entries $V_{\alpha_1,\alpha_2}$. Under this identification $U$ is identified with $V^*$. That is, the anti-unitary map $U$ may be identified with the map $V \mapsto V^*$. 
\end{remark}

\subsection{$QMS_m$ is always a pointed cone}

\begin{lm}\label{comm} An Hermitian operator $\cL$ satisfies
$\pm\cL\in QMS$   if and only if for some $H$ with $H = H^*$, $\cL(A) = i[H,A]$. 
\end{lm}

\begin{proof} It is evident that if $\cL(A) = i[H,A]$ with $H$ self-adjoint, then both $\cL$ and $-\cL$ belong to $QMS$.  For the converse, suppose that  $\cL$ and $-\cL$ belong to $QMS$. Consider some unital orthonormal basis $\{F_\alpha\}$  of $\mathfrak{H}$. Let $C_\cL$ be the characteristic matrix of $\cL$ with respect to this basis. Since $\cL$ and $-\cL$ belong to $QMS$, the reduced characteristic matrix of $\cL$ must be both positive semi-definite and negative semidefinite, and hence it is zero.  Thus $(c_\cL)_{\alpha,\beta} =0$ unless either $\alpha = (1,1)$ or $\beta = (1,1)$ or both. Since $\cL$ is Hermitian $(c_\cL)_{(1,1),\beta} = \overline{(c_\cL)_{\beta,(1,1)}}$.  Define
$$G := \frac12 (c_\cL)_{(1,1),(1,1)} + \sum_{\beta \neq (1,1)} (c_\cL)_{(1,1),\beta} F_\beta\ .$$
Then
${\displaystyle \cL(A) = \sum_{\alpha,\beta}(c_\Phi)_{\alpha,\beta}F_\alpha^* A F_\beta = G^*A + AG}$. Write $G = K-iH$ with $K$ and $H$ self-adjoint. Since $\cL(\one) = 0$, we have $G^*+G = 0$, that is, $K=0$. Then $\cL(A) = iHA-iAH = i[H,A]$.
\end{proof}

\begin{thm}\label{Pointed}  For any $m\in \mathcal{P}[0,1]$, the only $\Phi$ such that $\Phi\in QMS_m$ and $-\Phi\in QMS_m$ is $\Phi =0$.
\end{thm}

\begin{proof}  Suppose that $\Phi\in QMS_m$ and $-\Phi\in QMS_m$.  By Lemma~\ref{comm}, there is a self-adjoint $H\in \cM_N(\C)$ such that $\Phi(A) = i[H,A]$.  Let  $\{E_\alpha\}$ be the orthonormal basis of $\fH$ specified in \eqref{onbas}, so that we may apply Theorem~\ref{main1}.  We compute
\begin{eqnarray*}
(c_\Phi)_{\alpha,\beta} &=& i\frac{1}{N^2} \sum_{i,j} \tr \left[ (E_\alpha^* E_{i,j} E_\beta)^*(H E_{i,j} - E_{i,j}H)\right]\\
&=& i\frac{1}{N^2} \sum_{i,j} \tr \left( \tr[E_{i,j}E_\beta^*E_{j,i}E_\alpha H]  - \tr[R_{j,i}E_\alpha E_{i,j} H E_\beta^*]\right)\\
&=& i \delta_{\beta_1,\beta_2} \langle E^*_\alpha, H\rangle_{\mathfrak{H}}  -  i \delta_{\alpha_1,\alpha_2} \langle E_\beta, H\rangle_{\mathfrak{H}} \ .
\end{eqnarray*}
In particular, $(c_\Phi)_{\alpha,\beta} =0$ unless either $\alpha = \alpha'$ or $\beta = \beta'$ or  both. 

If $\alpha = \alpha'$ and $\beta = \beta'$, then 
$$
\langle E^*_\alpha, H\rangle_{\mathfrak{H}} = \frac1N\tr[E_\alpha H] = \frac1N\tr[ H E_\alpha ]  \in \R\ ,
$$
and likewise $\langle E_\beta, H\rangle_{\mathfrak{H}}\in \R$. But then $(c_\Phi)_{\alpha,\beta}$ is purely imaginary, and this contradicts \eqref{case1} unless $(c_\Phi)_{\alpha,\beta}= 0$. 

Next, suppose $\alpha = \alpha'$, but $\beta \neq \beta'$.  Then 
$$(c_\Phi)_{\alpha,\beta}  = -i\langle E_\beta, H\rangle_{\mathfrak{H}} \quad{\rm and}  \quad (c_\Phi)_{\beta',\alpha}  = i\langle E_\beta, H\rangle_{\mathfrak{H}}\ .$$
By \eqref{case2} we then have 
$$-(\lambda_{\alpha_2},\lambda_{\beta_2})_m \langle E_\beta, H\rangle_{\mathfrak{H}} =  (\lambda_{\alpha_2},\lambda_{\beta_1})_m \langle E_\beta, H\rangle_{\mathfrak{H}}\ ,$$
which is impossible unless $\langle E_\beta, H\rangle_{\mathfrak{H}}= 0$. But since $\beta$ is arbitrary apart from the condition that $\beta \neq \beta'$, this completes the proof that $H =0$.
\end{proof}

\subsection{The decomposition $\widehat{\mathfrak{H}} = \widehat{\mathfrak{H}}_{\mathcal{S}} \oplus \widehat{\mathfrak{H}}_{\mathcal{S}}^\perp$}

Recall Definition~\ref{orthsdef} of $\widehat{\mathfrak{H}}_{\mathcal{S}}$ as the subspace of all operators $\Phi$ on $\widehat{\mathfrak{H}}$ of the form $\Phi(A) = XA + AY$ for some $X,Y\in \cM_N(\C)$. 

\begin{lm}\label{orthchar} An operator  $\Phi$ on $\mathfrak{H}$ belongs to $\widehat{\mathfrak{H}}_{\mathcal{S}}$  if and only if for any unital orthonormal basis of $\mathfrak{H}$, the reduced characteristic matrix of $\Phi$ is zero; i.e., $R_\Phi =0$. 
\end{lm}

\begin{proof}  Let $\{F_\alpha\}$ be a unital orthonormal basis. Let $\{E_{i,j}\}$ be a matrix unit basis. Then for $\Phi(A) = XA + AY$ we compute
\begin{eqnarray}\label{gsw11}
(c_\Phi)_{\alpha,\beta} &=& \frac{1}{N^3}\sum_{i,j=1}^N \tr [ (F_\alpha^* E_{i,j}F_\beta)^*(XE_{i,j} + E_{i,j}Y)]\nonumber\\
 &=& \frac{1}{N^3}( \tr[F_\beta^*] \tr[F_\alpha X] + \tr[F_\alpha] \tr[F_\beta^*Y])\nonumber\\ 
 &=& \frac{1}{N}( \delta_{\beta, (1,1)} \overline{\langle F_{\alpha}, X\rangle_{\mathfrak{H}} } +  \delta_{\alpha,(1,1)} \langle F_{\beta},Y \rangle_{\mathfrak{H}})
\end{eqnarray} and this is zero unless either $\alpha= (1,1)$ or $\beta = (1,1)$ or both. Hence $R_\Phi =0$.

On the other hand, if $R_\Phi=0$, $\Phi$ has the form 
\begin{equation}\label{gsw1}
\Phi(A) = (c_\Phi)_{(1,1),(1,1)} +  \sum_{\alpha\neq (1,1)}(c_\Phi)_{\alpha,(1,1)}F_\alpha^*A  +   \sum_{\beta \neq (1,1)}(c_\Phi)_{(1,1),\beta }AF_\beta \ ,
\end{equation}
and hence if we define
\begin{eqnarray*}\label{gsw2}
X &:=& \frac12(c_\Phi)_{(1,1),(1,1)} \one     +   \sum_{\alpha\neq (1,1)}(c_\Phi)_{\alpha,(1,1) }F_\alpha^* \nonumber\\
Y &:=& \frac12(c_\Phi)_{(1,1),(1,1)} \one     +   \sum_{\beta \neq (1,1)}(c_\Phi)_{(1,1),\beta }F_\beta
\end{eqnarray*}
then $\Phi$ has the form
$\Phi(A) =   XA + AY$, and hence $\Phi\in \widehat{\mathfrak{H}}_{\mathcal{S}}$. 
\end{proof}

\begin{remark}\label{oneXY}  By the computation in \eqref{gsw11}, 
$$(c_\Phi)_{(1,1),(1,1)} = \frac1N \tr[X^*+Y]  = \overline{\langle E_{(1,1)},X\rangle_{\mathfrak{H}}} +\langle E_{(1,1)},Y\rangle_{\mathfrak{H}}\, $$ but for $\alpha\neq (1,1)$,  both $\langle X ,E_{\alpha} \rangle_{\mathfrak{H}} $  and $\langle Y ,E_{\alpha} \rangle_{\mathfrak{H}} $
can be read off from $C_\Phi$, so that if for some $X',Y'$ we also have $\Phi(A) = X'A + AY'$,  it follows that  for some $\eta\in \C$, $X' = X+\eta\one$ and $Y' = Y - \overline{\eta}\one$. 
\end{remark} 

\begin{lm}\label{redchar}  Let $\Phi$ be CP with a minimal Kraus representation $ \sum_{j=1}^M V_j^* A V_j$. Then the necessary and sufficient condition that $\Phi \in \widehat{\mathfrak{H}}_{\mathcal{S}}^\perp$
 is that $\tr[V_j] =0$ for each $j$.
\end{lm}

\begin{proof}  Let $\{F_\alpha\}$  be any unit orthonormal basis. By Lemma~\ref{orthchar}  $\Phi\in \widehat{\mathfrak{H}}_{\mathcal{S}}^\perp$ if and only if the characteristic and reduced characteristic matrices of $\Phi$, $C_\Phi$ and $R_\Phi$, respectively, are such that   the lower right $(N^2-1)\times (N^2-1)$ block of $C_\Phi$ is $R_\Phi$ and all other entries are zero.  In other words, the first row and column of $C_\Phi$ are zero .

Writing $V_j = \sum_\alpha S_{j,\alpha} F_\alpha$, for some unit orthonormal basis $\{F_\alpha\}$, 
${\displaystyle 
\Phi(A) = \sum_{\alpha,\beta}\left(\sum_{j=1}^M (S^*_{\alpha,j}S_{j,\beta} )\right) F_\alpha^* A F_\beta}$,
and hence
$$
(c_\Phi)_{\alpha,\beta} = \sum_{j=1}^M (S^*_{\alpha,j}S_{j,\beta})\ .
$$
If for each $j$, $\tr[V_j] =0$, then $S_{(1,1),j} = 0$ for each $j$, and hence the first row and column of $C_\Phi$ are zero. 

On the other hand, if the first row and column of $C_\Phi$ are zero, then 
$$0 = (c_\Phi)_{(1,1),(1,1)} = \sum_{j=1}^M |S_{j,(1,1)}|^2\ ,$$
and hence $S_{j,(1,1)} =0$ for all $j$.  This implies that $\tr[V_j] =0$. 
\end{proof}

\begin{lm}\label{orthinv}For all $s\in [0,1]$, the subspaces $\widehat{\mathfrak{H}}_{\mathcal{S}}$ and $\widehat{\mathfrak{H}}_{\mathcal{S}}^\perp$ are invariant under the operation of taking the adjoint on $\H_{\delta_s}$.
\end{lm}

\begin{proof} Evidently it suffices to prove the invariance of $\widehat{\mathfrak{H}}_{\mathcal{S}}$. Let $\Phi(A) = X A + AY$. Then $\Phi^\dagger(A) = Y^*A + AX^*$ and so 
\begin{equation}\label{gsw4}
\Phi^{*,\delta_s}(A) = (\Delta^{-s}Y^*)A + A(\Delta^{1-s}(G)X^*) \in \widehat{\mathfrak{H}}_{\mathcal{S}}\ .
\end{equation}
\end{proof}

\section{Self-adjointness for the KMS inner product}

It is now very easy to determine the structure of $CP_{KMS} := CP_{\delta_{1/2}}$. In this case, \eqref{mb4}
becomes
\begin{equation}\label{mb4KMS}
(b_{\Phi,m})_{\alpha,\beta}  =  \sqrt{\lambda_{\alpha_2}}(c_\Phi)_{\alpha,\beta} \sqrt{\lambda_{\beta_2}}\ ,
\end{equation}
Therefore, for Hermitian $\Phi$, $B_{\Phi,KMS}$ is positive semi-definite if and only if $C_\Phi$ is positive semidefinite; i.e., if and only if $\Phi$ is CP.  

\begin{thm}\label{Main2}  Let $\mathfrak{K}$ be the real vector space consisting of all $V\in \cM_N(\C)$ such that $\Delta^{-1/2} V = V^*$.  The extremal elements $\Phi$  of $CP_{KMS}$ are precisely the elements of the form 
\begin{equation}\label{KMSCP}
\Phi(A) = V^*AV\ , \qquad V\in \mathfrak{K}\ .
\end{equation}
Every map in $CP_{KMS}$ is a linear combination of at most $N^2$ such maps.
\end{thm}

 \begin{proof}
If $\Phi$ is extremal in the set of KMS self-adjoint CP maps, then necessarily $C_\Phi$  and $B_{\Phi,KMS}$ are rank one, because if $B_{\Phi,KMS}$ is not rank one, its spectral decomposition would allow it to be written as a sum of two positive matrices $B_1$ and $B_2$, neither a multiple of the other, and each commuting with $U$, the anti-unitary operator defined in  Theorem~\ref{main1}.  By Theorem~\ref{main1},  this would induce a decomposition of $\Phi$ into the sum of two KMS self-adjoint CP maps. It follows that $\Phi$ is extremal in the set of KMS self-adjoint CP maps if and only if 
$B = | u \rangle \langle u|$
where $u$ is an eigenfunction of $U$; i.e., $Uu = \pm u$. 
Then 
$$(c_\Phi)_{\alpha,\beta} =   \frac{1}{\sqrt{\lambda_{\alpha_2}}} \overline{u_\alpha} u_\beta  \frac{1}{\sqrt{\lambda_{\beta_2}}}$$
Thus if we define 
${\displaystyle 
V_u := \sum_\beta \frac{1}{\sqrt{\lambda_{\beta_2}}} u_\beta E_\beta
}$,
then $\Phi(A) = V_u^*AV_u^{\phantom{*}}$.   

Next, since 
$$\Delta^{-1/2}E_\beta = \frac{\sqrt{\lambda_{\beta_2}}}{\sqrt{\lambda_{\beta_1}}}E_\beta\ ,
$$
and hence, since $\overline{u_\beta} = \pm u_{\beta'}$, 
$$
\Delta^{-1/2}V_u^{\phantom{*}} = \sum_\beta \frac{1}{\sqrt{\lambda_{\beta_1}}}u_\beta  E_\beta = \pm \sum_\beta \frac{1}{\sqrt{\lambda_{\beta_1}}}  \overline{ u_{\beta'}} E_\beta
= \pm \sum_\beta \frac{1}{\sqrt{\lambda_{\beta_2}}}  \overline{ u_{\beta}} E_{\beta'}  = \pm V_u^*\ .
$$
Now suppose that $V$ is such that $\Delta^{-1/2}V = \pm V^*$, and $\Phi(A) = V^*AV$.  Then  $\Phi^\dagger (A) = VAV^*$, and then by Lemma~\ref{adlem},
$$\Phi^{*,KMS}(A) = \sigma^{-1/2}(V(\sigma^{1/2} A \sigma^{1/2})V^*)\sigma^{-1/2} = (\Delta^{-1/2}V)^* A \Delta^{1/2}V = V^*AV = \Phi(A)\ .$$

Finally, observe that if $\Delta^{-1/2}V = - V^*$, then $\Delta^{-1/2}(iV) = (iV)^*$ and of course $V^* AV = (iV)^* A (iV)$, so we need only  concern ourselves with $V$ such that $\Delta^{-1/2}V = V^*$.

Now  equip $\mathfrak{K}$ with the inner product
$$
\langle V,W\rangle_{\mathfrak{K}} =  \Re\left( \langle V,W\rangle_{\mathfrak{H}}\right)\ .
$$

We will make use of the following  orthonormal basis for $\mathfrak{K}$:
Let $\{E_\alpha\}$ be a matrix unit basis associated to $\sigma$.  First define
\begin{equation}\label{omegadef}
\omega_\alpha := \log \lambda_{\alpha_1} - \log \lambda_{\alpha_2}\ ,
\end{equation}
and note that $\Delta^{1/2}E_\alpha = e^{\omega_\alpha/2}E_\alpha$.
 For $\alpha = (\alpha_1,\alpha_2)$ such that $\alpha_1 < \alpha_2$, define
\begin{equation}\label{GB1_KMS}
G_\alpha := \frac{1}{i\sqrt{2\cosh(\omega_\alpha/2)}}( e^{\omega_\alpha /4} E_\alpha  -   e^{-\omega_\alpha /4} E_\alpha^*)\ .
\end{equation}
For $\alpha = (\alpha_1,\alpha_2)$ such that $\alpha_1 \geq  \alpha_2$, define
\begin{equation}\label{GB2_KMS}
G_\alpha :=  \frac{1}{\sqrt{2\cosh(\omega_\alpha/2)}} e^{\omega_\alpha /4} E_\alpha  +   e^{-\omega_\alpha /4} E_\alpha^*\ .
\end{equation}
Then one readily checks that  for all $\alpha$, 
\begin{equation}\label{GB3_KMS}
 \Delta^{-1/2} (G_\alpha ) =  G_\alpha^*\ ,
\end{equation}
and that 
$\{G_\alpha\}_{\alpha\in \cJ_N}$ is orthonormal in $\mathfrak{K}$.

We now show that  $\{G_\alpha\}_{\alpha\in \cJ_N}$  is a basis for the real linear space in question. 
Suppose that $V \in \mathfrak{K}$. Let $\{E_\alpha\}$ be the modular basis out of which the orthonormal basis $\{G_\alpha\}$ was constructed.  We then expand
 $$V = \sum_{\alpha} a_\alpha e^{\omega_\alpha/4} E_\alpha\ .$$
 We compute
 $$\Delta^{-1/2}(V)  = \sum_{\alpha} a_\alpha e^{-\omega_\alpha/4} E_\alpha
 \qquad{\rm and}\qquad  
 V^* = \sum_{\alpha} \overline{a_\alpha} e^{\omega_\alpha/4} E_\alpha^*  = 
 \sum_{\alpha} \overline{a_{\alpha'}}e^{\omega_{\alpha'}/4} E_{\alpha}\ . $$
 It follows that since  $\Delta^{-1/2}(V) = V^*$, then $a_\alpha e^{-\omega_\alpha/4} = \overline{a_{\alpha'}}e^{\omega_{\alpha'}/4} =
 \overline{a_{\alpha'}}e^{-\omega_{\alpha}/4}$, and hence $a_\alpha = \overline{a_{\alpha'}}$.    Therefore, if $a_\alpha = x_\alpha+iy_\alpha$
 is the decomposition of $a_\alpha$ into its real and imaginary parts,
 \begin{equation*}
 a_\alpha e^{\omega_\alpha/4} E_\alpha +  
 a_{\alpha'} e^{\omega_{\alpha'}/4} E_{\alpha'}  = x_\alpha[ e^{\omega_\alpha/4} E_\alpha + e^{-\omega_\alpha/4} E_\alpha^*] 
 + iy_\alpha[ e^{\omega_\alpha/4} E_\alpha - e^{-\omega_\alpha/4} E_\alpha^*]
\end{equation*}
 Now consider any $\alpha = (\alpha_1,\alpha_2)$ with $\alpha_1 > \alpha_2$.   Then we have 
 $$ a_\alpha e^{\omega_\alpha/4} E_\alpha +  
 a_{\alpha'} e^{\omega_{\alpha'}/4} E_{\alpha'}  =  x_\alpha G_\alpha + y_\alpha G_{\alpha'}\ ,$$
 while if $\alpha_1 = \alpha_2$, $a_\alpha = a_{\alpha'} =x_\alpha$ and $a_\alpha e^{\omega_\alpha/4} E_\alpha +  
 a_{\alpha'} e^{\omega_{\alpha'}/4} E_{\alpha'} =2x_\alpha E_\alpha = x_\alpha G_\alpha$.  Hence $V$ is a real linear combination of the 
 $\{G_\alpha\}$. 
\end{proof}

\begin{thm}\label{RNlem}
Let $\Phi$ and $\Psi$ be two  CP maps that are KMS self-adjoint.  Let $\Phi(A) =  \sum_{j=1}^M V_j^*A V_j$ be a minimal Kraus representation of $\Phi$ with each $V_j\in \mathfrak{K}$. Then $\Phi - \Psi$ is CP if and only if there exists  a real  $M\times M$ matrix $T$  such that $0 \leq T \leq \one$ and
\begin{equation}\label{baskrid3}
 \Psi(A) = \sum_{i,j=1}^M T_{i,j} V_i^* A V_j\ .
\end{equation}
\end{thm}

\begin{proof} Suppose that  $\Phi - \Psi$ is CP.  Then by Arveson's Theorem, there exists a uniquely determined  $T$ with $0 \leq T \leq \one$ such that \eqref{baskrid3} is valid. Since $\Psi$ is KMS self-adjoint,
$$\Psi(A) = \sum_{i,j=1}^M \overline{T_{i,j}} \sigma^{-1/2}  V_i \sigma^{1/2}A \sigma^{1/2} V_j^j \sigma^{-1/2} = \sum_{i,j=1}^M \overline{T_{i,j}} V_i^* A V_j\ .
$$
By the uniqueness, $T$ is real.

Conversely, suppose that $\Psi$ has the form specified in  \eqref{baskrid3} with $0 \leq T \leq \one$ and $T$ is  real.  Then $\Phi$ is CP and KMS self-adjoint, and since $T \geq 0$, $\Psi$ is CP, 
and  since $T \leq \one$, $\Phi \geq \Psi$.  Finally, since $T$ is real, $\Psi$ is KMS self-adjoint.
\end{proof}

\begin{thm} Let $\Phi$ be a unital CP map that is  KMS self-adjoint and let $\Phi(A) = \sum_{j=1}^M V_j^*A V_j$ be a minimal Kraus representation of $\Phi$ with $V_j\in \mathfrak{K}$ for each $j$.  Then $\Phi$ is an extreme point of the set of unital CP maps that are KMS self-adjoint if and only if
\begin{equation}\label{liset}
\{ V_i^*V_j + V_j^*V_i\ :\ 1\leq i \leq j \leq M\}
\end{equation}
 is linearly independent over the real numbers.   
\end{thm}

\begin{proof}  Suppose that the set in \eqref{liset} is linearly independent over the real numbers. Suppose that $\Psi$ is unital, CP and KMS self-adjoint, and that for some $0 < t < 1$, $\Phi -t\Psi$ is CP. We must show that $\Psi=\Phi$. 
By Lemma~\ref{RNlem}, there is a real $M\times M$ matrix $T$  such that $0 \leq T \leq \one$ and 
$$t\Psi(A) = \sum_{i,j=1}^M T_{i,j} V_i^* A V_j\ .$$
Then
$$t\one = t\Psi(\one) = \sum_{i,j=1}^M T_{i,j} V_i^*  V_j \qquad{\rm and}\qquad 
t\one = t\Phi(\one) = \sum_{i,j=1}^M t\delta_{i,j} V_i^*V_j\ .$$
Therefore
$$
0 = \sum_{i,j=1}^M ( t\delta_{i,j} - T_{i,j}) V_i^*V_j\ .
$$
Then by the linear independence, $t\delta_{i,j} - T_{i,j} =0$ for each $i,j$. Thus $t\Psi = t\Phi$, and $\Phi$ is extreme. 

Now suppose that   the set in \eqref{liset} is not linearly independent over the reals. Then there is an $M\times M$ real symmetric matrix $B$ such that $B_{i,j}$ are not all zero and $\sum_{i,j} B_{i,j}V_i^*V_j = 0$. Choose some $t > 0$ such that
$$ 0 \leq T :=  \one + tB \leq 2\one\ ,$$
and define $\Psi$ by 
$$
\Psi(A) = \sum_{i,j=1}^M T_{i,j} V_i^* A V_j\ .
$$
Then $\Psi$ is CP, KMS self-adjoint, and 
$$\Psi(\one) = \sum_{i,j=1}^M (\delta_{i,j} + tB_{i,j})V_i^* A V_j = \sum_{i,j=1}^M \delta_{i,j} V_i^* \one V_j =\Phi(\one) = \one \ .$$
Thus $\Psi$ is unital and since $\Psi \leq \frac12 \Phi$, $\Phi$ is not extreme. 
 \end{proof}
 
 \subsection{$QMS_{KMS}$}
 
The problem of determining the structure of QMS generators that are self-adjoint on $\H_{KMS}$ has been studied in \cite{FU07}, but we give a simpler approach that goes further in one important aspect. 

Let $\cL\in QMS_{KMS}$ and  write it in the form 
\begin{equation}\label{gsw51}
\cL(A) = (G^*A + AG) +  \sum_{j=1}^M V_j^*A V_j
\end{equation}
where $\{V_1,\dots,V_M\}$ is linearly independent, which we may always do for any QMS generator. Replacing each $V_j$ by $V_j - \tr[V_j]\one$ and absorbing the difference into $G$, 
we may assume without loss of generality that for all $j$ $\tr[V_j] =0$.  By Lemma~\ref{redchar},  $\Psi$, given by $\Psi(A) = \sum_{j=1}^M V_j^* A V_j$ belongs to
$\widehat{\mathfrak{H}}_{\mathcal{S}}^\perp$.  Under these conditions on $\{V_1,\dots,V_M\}$,  \eqref{gsw51} is the orthogonal decomposition of $\cL$ into its 
components in $\widehat{\mathfrak{H}}_{\mathcal{S}}$ and $\widehat{\mathfrak{H}}_{\mathcal{S}}^\perp$.

By Lemma~\ref{orthinv}, each component is individually self-adjoint.   Let $\Phi$ denote the map $\Phi(A) = G^*A + AG$. Then $\Phi^\dagger(A) =  GA + AG^*$ and hence
$$
\Phi^{*,KMS}(A) = (\Delta^{-1/2} G) A  + A (\Delta^{1/2} G^*) 
$$

Furthermore, adding a purely imaginary multiple of $\one$ to $G$ has no effect on the operation $A \mapsto G^*A + AG$, and hence we may assume without loss of generality that $\tr[G] \in \R$.  
Then $\tr[\Delta^{-1/2}G]\in \R$, and by Remark~\ref{oneXY}, we must have $\Delta^{-1/2}G = G^*$.

Thus, a QMS generator $\cL$  is self-adjoint on $\H_{KMS}$ if and only if it can be written in the form
\begin{equation}\label{gsw14}
\cL(A) = G^*A + A G +  \sum_{j=1}^M V_j^* A V_j
\end{equation}
where $\tr[G]\in \R$, $\{V_1,\dots,V_M\}$ is linearly independent, $\Delta^{-1/2}G = G^*$ and $\Delta^{-1/2}V_j = V_j^*$ as well as $\tr[V_j] =0$ for each $j$. 
As explained in the introduction, this much one finds in \cite{FU07}. However, there is a compatibility question to be dealt with. Since $\cL(\one) =0$, we must have 
\begin{equation}\label{gsw15}
0 = G^* +  G +  \sum_{j=1}^M V_j^* V_j\ .
\end{equation}
That is, writing $G = H+iK$, $H$ and $K$ self-adjoint,
\begin{equation}\label{gsw16VastV}
H = \frac12 \sum_{j=1}^M V_j^* V_j\ .
\end{equation}
Now in general $\Delta^{-1/2}(H)$ is not even self-adjoint, so that $\Delta^{-1/2}(H) = H$ does not generally hold true.

This raises the following question:  Given any CP map $\Psi$ with a minimal Kraus representation $\Psi(A) = \sum_{j=1}^M V_j^* A V_j$ such that for each $j$ $\Delta^{-1/2}V_j = V_j^*$ and $\tr[V_j] =0$, when does there exist a $G$ such that 
\begin{equation}\label{gsw16}
\cL(A) := G^*A + AG + \Psi(A)
\end{equation}
belongs to $QMS_{KMS}$?
By what has been noted above, we may as well require $\tr[G]\in \R$, and then we  must have $\Delta^{-1/2}G = G^*$, and we must have $G = H+iK$ , $H$ and $K$ self-adjoint, with $H$ specified by \eqref{gsw16VastV}.   It turns out that there always exists a unique choice of $K$ such that $\Delta^{-1/2}G = G^*$, and thus the answer to the question just raised is ``always''. 

\begin{thm}\label{Main7}  There is a one-to-one correspondence between elements $\cL$ of  $QMS_{KMS}$  and CP maps $ \Psi\in \widehat{\mathfrak{H}}_\mathcal{S}^\perp$ that are self-adjoint on $\H_{KMS}$.   
The correspondence identifies $\Psi$ with $\cL_\Psi$ where
\begin{equation}\label{gsw16}
\cL_\Psi(A) = G^*A + AG + \Psi(A)
\end{equation}
where $G = H+iK$, $H$ and $K$ self-adjoint and given by
\begin{equation}\label{HforKMS}
H := \tfrac12\Psi(\one)
\end{equation}
and 
\begin{equation}\label{gsw17}
K := \frac1i\int_0^\infty e^{-t   \sigma^{1/2}} ( \sigma^{1/2}  H - H  \sigma^{1/2})  e^{-t   \sigma^{1/2}}{\rm d} t\ .
\end{equation}
Furthermore,  for all $\cL_{\Psi_1},\cL_{\Psi_2} \in QMS_{KMS}$,  $\cL_{\Psi_1} - \cL_{\Psi_2} \in QMS$ if and only if $\Psi_1 - \Psi_2 \in CP$; i.e., $\cL_{\Psi_1} \geq \cL_{\Psi_2}$ if and only if 
$\Psi_1 \geq \Psi_2$, and the extreme points of $QMS_{KMS}$ are precisely the generators of the form
$$\cL(A) := G^*A + AG + V^*AV$$
where $\Delta^{-1/2}V = V^*$ and where $G=H+iK$ is given by \eqref{HforKMS} and \eqref{gsw17} for $\Psi(A) = V^*AV$. 
\end{thm} 

\begin{proof}  Consider such a set $\{V_1,\dots V_M\}$ and define 
\begin{equation}
\Psi(A) = \sum_{j=1}^M V_j^* A V_j\ .
\end{equation}
Since $\tr[V_j] = 0$ for each $j$, $ R_\Psi$ is positive semidefinite, and since $\Delta^{-1/2}V_j = V_j^*$ for each $j$,  $\Psi$ is self-adjoint on $\H_{KMS}$.
Then for any choice of $G$, $\cL(A) = G^*A + AG + \Psi(A)$ generates a CP semigroup, and it belongs to $QMS$ if and only if $\cL(\one) =0$, and this is the case if and only if the self-adjoint part of $H := \tfrac12(G+G^*)$ is given by \eqref{gsw16}.
Finally, $\cL$ will be self-adjoint on $\H_{KMS}$ if and only if $K := -i\tfrac12 (G - G^*)$ is  chosen so that $\Delta^{-1/2}(G) = G^*$. 

The equation $\Delta^{-1/2}(G) = G^*$ is equivalent to $(H + i K)\sigma^{1/2} = \sigma^{1/2} (H - i K)$, and rearranging terms we have 
\begin{equation}\label{gsw31}
\sigma^{1/2}  K + K \sigma^{1/2}  = -i( \sigma^{1/2}  H - H  \sigma^{1/2})\ .
\end{equation}
This is a Lyapunov equation, and  for any $X$ self-adjoint, the unique solution $K$ of $\sigma^{1/2}  K + K \sigma^{1/2} = X$ is given by
$$K := \int_0^\infty e^{-t   \sigma^{1/2}} X e^{-t   \sigma^{1/2}}{\rm d} t\ .$$
Indeed,
$$\sigma^{1/2} \left(  \int_0^\infty e^{-t   \sigma^{1/2}} X e^{-t   \sigma^{1/2}}{\rm d} t  \right)   \sigma^{1/2}  = -\int_0^\infty \frac{{\rm d}}{{\rm d}t}e^{-t   \sigma^{1/2}} X e^{-t   \sigma^{1/2}}{\rm d} t = X\ .$$
Thus we must take $K$ to be given by \eqref{gsw17}.   Note that $\tr[K]=0$ as an easy consequence of cyclicity of the trace, and hence $\tr[G]\in \R$.

Now suppose $\cL_{\Psi_1},\cL_{\Psi_2} \in QMS_{KMS}$. Then  $\cL_{\Psi_1} - \cL_{\Psi_2} \in QMS$ if  and only if its reduced density matrix is positive semidefinite. By Lemma~\ref{orthchar} and Lemma~\ref{redchar}, this is the case if and only if $\Psi_1-\Psi_2$ is CP.  The final assertion now follows from Theorem~\ref{Main2}.
\end{proof}

 \section{Self-adjointness for the GNS inner product}
 
 The other case of main interest that is easily handled is the case of self-adjointness for the GNS inner product. While the KMS inner product corresponds to the measure $m = \delta_{1/2}$, the GNS inner product corresponds to the measure $\delta_0$. 
 It turns out that we may as well consider $m = \delta_s$, $s \neq \tfrac12$; the set of CP maps that are self-adjoint on $\H_{\delta_s}$ does not depend on $s\neq \frac12$.  The structure of $QMS_{GNS}$ was worked out by Alicki, and his methods adapt well to the study of $CP(\one)_{GNS}$.
 
 \begin{thm}\label{Main3}  The extremal rays in the cone of  CP maps that are self-adjoint on $\H_{\delta_s}$, $s\neq \frac12$,  are precisely the  maps of the form
\begin{equation}\label{SmsCP}
\Phi(A) = e^{\omega/2}V^* AV +  e^{-\omega/2}V A V^*
\end{equation}
where $\Delta V = e^{\omega} V$, $\omega > 0$, or of the form
\begin{equation}\label{SmsCP2}
\Phi(A) = VAV \ ,
\end{equation}
$\Delta V = V = V^*$
  In particular, every CP map that is self-adjoint on $\H_{\delta_s}$ is a positive linear combination of the operators specified in \eqref{SmsCP} and \eqref{SmsCP2}.
\end{thm}

\begin{remark}  The restriction to $\omega > 0$ in \eqref{SmsCP} is not essential; it is to avoid double counting, since the eigenvalues $e^\omega$  of $\Delta$ with $\omega>0$ and with $\omega< 0$ enter only in the precisely paired manner specified in 
\eqref{SmsCP}.  Also, note that replacing $V$ by $e^{i\theta} V$, $\theta\in \R$, has no effect on the map in \eqref{SmsCP2}. There do exist $V$ that are not self-adjoint such that $A\mapsto V^*AV$ is self-adjoint on $\H_{\delta_s}$ and extremal even in the wider class of all CP maps. However, the theorem asserts that in this case one may replace $V$ by an appropriate complex multiple and then $V = V^*$.
\end{remark}

\begin{proof} Specializing to $m= \delta_s$,  \eqref{mb3} becomes
 \begin{equation}\label{mb3GNS}
 (c_\Phi)_{\alpha,\beta} = \frac{\lambda_{\alpha_1}^s\lambda_{\beta_1}^{1-s}}{\lambda_{\alpha_2}^s\lambda_{\beta_2}^{1-s}}  (c_\Phi)_{\beta',\alpha'} = e^{s\omega_\alpha  + (1-s)\omega_\beta} (c_\Phi)_{\beta',\alpha'} \ .
\end{equation}
As a consequence, $(c_\Phi)_{\beta',\alpha'}  = e^{s\omega_{\beta'}+(1-s)\omega_{\alpha'}}(c_\Phi)_{\alpha,\beta}  =  e^{-s\omega_{\beta}-(1-s)\omega_{\alpha}}(c_\Phi)_{\alpha,\beta}$. Altogether, 
\begin{equation}\label{baskrid5}
 e^{(1-2s)\omega_{\alpha}}(c_\Phi)_{\alpha,\beta} =  (c_\Phi)_{\alpha,\beta} e^{(1-2s)\omega_\beta}\ .
\end{equation}
That is, $C_\Phi$ commutes with the diagonal matrix $\Omega$ whose $\alpha$th diagonal entry is $e^{(1-2s)\omega_\alpha}$.  Of course this condition is vacuous for $s=\tfrac12$, but otherwise it is a strong restriction on $C_\Phi$.
  Let us order the entries in such a manner that all indices $\alpha$ for which $e^{\omega_\alpha}$ has the same value are grouped together. Then $C_\Phi$ will have a block structure. In any case,
  as a consequence of \eqref{baskrid5}, for $s\neq \frac12$,
  \begin{equation}\label{baskrid6}
 \omega_\alpha \neq \omega_\beta \ \Rightarrow \  (c_\Phi)_{\alpha,\beta} =0\ .
\end{equation}

The blocks correspond to the distinct eigenvalues of the modular operator. Let $\mu$ be such an eigenvalue and let $\cJ_\mu = \{\alpha \ :\ e^{\omega_\alpha} = \mu\ \}$. Apart from $\mu =1$, which is always an eigenvalue, the eigenvalues come in pairs. Let $\mu' := \frac{1}{\mu}$; then $\alpha \mapsto \alpha'$ is a one-to-one map from $\cJ_\mu$ onto $\cJ_{\mu'}$. For $\alpha,\beta\in \cJ_\mu$, \eqref{mb3GNS} reduces to
$ (c_\Phi)_{\alpha,\beta} = \mu^{(1-2s)}(c_\Phi)_{\beta',\alpha'} =\mu^{(1-2s)}\overline{(c_\Phi)_{\alpha',\beta'}} $, or equivalently,
 \begin{equation}\label{baskrid7}
 \mu^{-(1-2s)/2}(c_\Phi)_{\alpha,\beta}   = \mu^{(1-2s)/2}\overline{(c_\Phi)_{\alpha',\beta'}} \ .
\end{equation}
Hence, for $\mu \neq 1$,  if $U_{\alpha,\beta}$,  $\alpha,\beta\in \cJ_\mu$ is a unitary that diagonalizes the block corresponding to the eigenvalue $\mu$, $\overline{U_{\alpha',\beta'}}$ is a unitary that diagonalizes the block corresponding to the eigenvalue $\frac{1}{\mu}$.  For $\mu =1$,  \eqref{mb3GNS} further reduces to  $(c_\Phi)_{\alpha,\beta}  =\overline{(c_\Phi)_{\alpha',\beta'}}$ which means that this block of $C_\Phi$ is diagonal in an orthonormal basis consisting of eigenvectors of this anti-unitary transformation.  The eigenvectors $v$  are such that  
$$\left(\sum_{\alpha\in \cJ_1} v_\alpha E_\alpha\right)^*  = \sum_{\alpha\in  \cJ_1} \overline{v_\alpha} E_\alpha' =\pm \sum_{\alpha\in  \cJ_1} v_\alpha E_\alpha'  = \pm \sum_{\alpha\in \cJ_1} v_\alpha E_\alpha\ .$$
 commutes with the anti-unitary transformation $v_\alpha \mapsto \overline {v_{\alpha'}}$. Replacing $v$ by $-v$ as needed, we can arrange that $\sum_{\alpha\in \cJ_1} v_\alpha E_\alpha$ is self-adjoint.  
 Then the $\mu =1$ block of $C_\Phi$ can be diagonalized by a unitary  $U_{\gamma,\beta}$ such that for each $\gamma\in \cJ_1$,  $\sum_{\beta\in \cJ_1}U_{\gamma,\beta}E_\beta$ is self-adjoint. 
 
Now we piece together all the unitary blocks into an $N^2\times N^2$ unitary that we still call $U$, being careful to use the ``matched'' unitaries in the adjoint blocks, as described above.

 It follows that there is a unitary matrix $U$ and non-negative numbers $c_\gamma$ such that
 $$c_{\alpha,\beta} = \sum_{\gamma} c_\gamma U_{\alpha,\gamma}^* U_{\gamma,\beta} = 
 \sum_{\gamma} c_\gamma \overline{U_{\gamma,\alpha}} U_{\gamma,\beta}  $$
 and such that  $U_{\gamma,\beta} =0$ unless $e^{\omega_\gamma} = e^{\omega_\beta}$, and defining 
 ${\displaystyle V_\gamma := \sum_{\beta}  U_{\gamma,\beta}E_\beta}$ we have
 $$
 \Delta V_\gamma = e^{\omega_\gamma}V_\gamma
 $$
 for all $\gamma$ and 
 $$
 \{ V_\gamma\ :\ \gamma\in \cJ_N\} =  \{ V_\gamma^*\ :\ \gamma\in \cJ_N^*\} 
 $$
 is an orthonormal basis of $\mathfrak{H}$.   When working with this basis, we redefine the map $\alpha \mapsto \alpha'$  by $V_{\alpha'} = V_\alpha^*$. This can only differ from the definition used with the matrix unit bases for indices in $\cJ_1$, and only if $\sigma$ has at least two eigenvalues equal to one another, producing an ``accidental'' eigenvector of $\Delta$ with the eigenvalue $1$. 
 
 Then $\Phi(A) = \sum_{\alpha,\beta} c_{\alpha,\beta} E_\alpha^* A E_\beta$ becomes 
 ${\displaystyle \Phi(A) := \sum_{\gamma} c_\gamma V_\gamma^* A V_\gamma}$.
 Since  $\Delta_\sigma(V_\gamma) = e^{\omega_\gamma}V_\gamma$ and $V_\alpha^* = V_{\alpha'}$
 \begin{eqnarray*}
 \Phi^{*,m_s}(A) &=&  \sum_{\gamma} \sigma^{-s} V_\gamma  \sigma^s A\sigma^{1-s} V_{\gamma}^* \sigma^{-(1-s)} = \sum_{\gamma} c_\gamma e^{-s\omega_\gamma}e^{(s-1)\omega_\gamma} V_\gamma A V_{\gamma}^*\\
   &=& 
  \sum_{\gamma} c_{\gamma'} e^{\omega_\gamma} V_\gamma^* A V_{\gamma} \ .
 \end{eqnarray*}
 By the uniqueness of the coefficients, $ \Phi^{*,m_s} = \Phi$ if and only if for each $\gamma$
  \begin{equation}\label{baskrid7}
 c_\gamma = c_{\gamma'} e^{\omega_\gamma}\ .
\end{equation}
Defining $b_\gamma = e^{-\omega_\gamma/2}c_\gamma$, \eqref{baskrid7} is equivalent to $b_\gamma = b_{\gamma'}$.   Let $\mathcal{S}$ denote the spectrum of $\Delta$, which is of course determined by the spectrum of $\sigma$.
We can finally write $\Phi$ in the form 
 \begin{equation}\label{baskrid7b}
 \Phi(A) = \sum_{\gamma\in \cJ_1}  c_\gamma V_\gamma A V_\gamma  +  \sum_{\mu\in \mathcal{S} ,\mu > 1}\sum_{\gamma\in \cJ_\mu}  b_\gamma \left(  e^{\omega_\gamma/2}V_\gamma^*A V_\gamma +  e^{-\omega_\gamma/2}V_\gamma A V_\gamma^*  \right) \ .
\end{equation} 

It remains to show that the maps specified in \eqref{SmsCP} are extremal.  Suppose first that  $V = e^{i\theta}V^*$ for some real $\theta$. Then $\omega = 0$, and \eqref{SmsCP} reduces to $\Phi(A) = V^*AV$ which is extremal 
 in the larger cone of all CP maps, and thus extremal among those that are self-adjoint on $\H_{\delta_s}$. Replacing $V$ by $e^{-i\theta/2}V$ we see that we may assume without loss of generality in this case that $V = V^*$. 
 
 Next, suppose that $V \neq e^{i\theta} V^*$ for any real $\theta$, but $\omega = 0$. Write $V = X+iY$,  $X,Y$ self-adjoint.  Then $\Delta X = X$ and $\Delta Y = Y$ and $X \neq Y$. We compute
 \begin{eqnarray*}
 V^*AV &=& (X-iY)A(X+iY) = XAX + YAY +i(XAY - YAX)\\
   VAV^* &=& (X+iY)A(X-iY) = XAX + YAY -i(XAY - YAX)  \ .
 \end{eqnarray*}
Thus   $\Phi(A) = XAX + YAY$.
Then $A \mapsto XAX$ and $A \mapsto YAY$ are distinct CP maps that are self-adjoint on $\H_{\delta_s}$.  Hence $\Phi$ is not extreme.   In summary, when $\Delta V = V$, the necessary and sufficient condition for $\Phi$ to be extremal is that $V^* = e^{i\theta }V$ for some real $\theta$, in which case we may replace $V$ by an equivalent self-adjoint operator. 

Now suppose that $\Delta V = e^\omega V$ with $\omega \neq 0$. Then $V$ and $V^*$ are orthogonal.  For convenience in what follows define $W_1 := e^{-(1-2s)\omega/4}V$ and $W_2 := e^{(1-2s)\omega/4}V^*$. Then,
ignoring the trivial case $V =0$, $$\Phi(A) = W_1^*AW_1 + W_2^*AW_2 $$
is a minimal Kraus representation of $\Phi$. Consequently, if $\Psi$  is any CP map such that  $\Phi - \Psi$ is CP, then $\Psi$ has the form
$$\Psi(A) = \sum_{i,j} T_{i,j} W_i^*A W_j$$
where $T$ is a $2\times 2$ matrix such that $0 \leq T \leq \one$.   Then,
$$
\Psi^{*,\delta_s}(A) = \sum_{i,j} \overline{T_{i,j}} (\Delta^{-s} W_i)A (\Delta^{1-s}W_j^*) =  \sum_{i,j} \overline{T_{i,j}} e^{-s(3-2i)\omega}e^{(1-s)(3-2j)\omega} W_iA W_j^*\ .
$$
Since  $\overline{T_{i,j}} = T_{j,i}$, simple calculations yield
$$
\Psi^{*,\delta_s}(A)   =  T_{1,1}  W_2^*A W_2  +  T_{2,2}  W_1^*A W_1  +  T_{2,1} e^{-\omega}W_1AW_2^* + T_{1,2} e^{\omega}W_2AW_1^* \ . $$
Then with $\Psi^{*,\delta_s} = \Psi$, by the uniqueness of the coefficients, $T_{1,2} = T_{2,1} =0$ and $T_{1,1} = T_{2,2}$. This shows that $T$ is a multiple of the identity, and hence that $\Psi$ is a multiple of $\Phi$. Hence $\Phi$ is extreme. 
\end{proof}

The following notation will be useful going forward. For $\mu \in \mathcal{S}$, let $d_\mu$ be the dimension of the corresponding eigenspace of $\Delta$.   As a consequence of Theorem~\ref{Main3}, every CP map $\Phi$ that is self-adjoint on $\H_{\delta_s}$ has a Kraus representation  of the following form:

Let $\mathcal{S}'= \{\mu_1,\dots,\mu_d\}$ be a subset of $\mathcal{S}$ such that
$1 \leq \mu_1 \leq \mu_2 \cdots  \leq \mu_d$. For each $1\leq j \leq d$, let $\{V^{(j)}_1,\dots, V^{(j)}_{M_j}\}$ be a linearly independent set of eigenvectors of $\Delta$ with eigenvalue $\mu_j$. Evidently, $1 \leq M_{j} \leq d_j$. Suppose further that if $\mu_1 =1$, then $\{V^{(1)}_1,\dots, V^{(1)}_{M_1}\}$ is self-adjoint.   Then, if $d_1 >1$, the Kraus representation is
\begin{equation}\label{canform1}
\Phi(A) =  \sum_{j=1}^d\left( \sum_{k = 1}^{M_j} ( \mu^{1/2}_k V^{(j)*}_k A V_k^{(j)}  +   \mu_j^{-1/2}V^{(j)}_k A V^{(j)^*}_k)\right)\ ,
\end{equation}
while if $d_1 =1$, it is 
 \begin{equation}\label{canform2}
\Phi(A) = \sum_{k = 1}^{M_1} V^{(1)}_k A V^{(1)}_k  + \sum_{j=2}^d\left( \sum_{k = 1}^{M_j} ( \mu_j^{1/2}V^{(j)*}_k A V_k^{(j)}  +   \mu_j^{-1/2}V^{(j)}_k A V_k^{(j)^*})\right)\ .
\end{equation}
In either case, the minimality of the Kraus representation is a consequence of the linear independence required above. We call such a minimal Kraus representation of a CP map that is self-adjoint on $\H_{\delta_s}$, $s\neq \frac12$ a {\em canonical minimal Kraus representation}.

\begin{thm}\label{Main4}  Let $\Phi$ be a CP map that is self-adjoint on $\H_{\delta_s}$, $s\neq \frac12$. Then if $\Psi$ is another CP map  that is self-adjoint on $\H_{\delta_s}$,   $\Phi- \Psi$ is CP if and only if:

\smallskip
\noindent{\it (1)} If $\Phi$ has a canonical minimal Kraus representation of the form \eqref{canform1}, there are matrices $T^{(1)},\dots, T^{(d)} $ where $T^{(j)}$ is an $M_j\times M_j$ matrix satisfying $0 \leq T^{(j)} \leq \one$ such that 
\begin{equation}\label{Main41}
\Psi(A) =   \sum_{j=1}^d\left( \sum_{k,\ell = 1}^{M_j}  T^{(j)}_{k,\ell}  \left( \mu_j^{1/2} V^{(j)*}_k A V_\ell^{(j)}  +   \mu_j^{-1/2}  V^{(j)}_\ell A V^{(j)^*}_k  \right)\right)\ ,
\end{equation}

\smallskip
\noindent{\it (2)} If $\Phi$ has a canonical minimal Kraus representation of the form \eqref{canform2}, there are matrices $T^{(1)},\dots, T^{(d)} $ where $T^{(j)}$ is an $M_j\times M_j$ matrix satisfying $0 \leq T^{(j)} \leq \one$, and with $T^{(1)}$ real, such that 
\begin{equation}\label{Main42}
\Psi(A) =   \sum_{k,\ell = 1}^{M_1} T^{(1)}_{k,\ell}V^{(1)}_k A V^{(1)}_\ell  +  \sum_{j=2}^d\left( \sum_{k,\ell = 1}^{M_j}  T^{(j)}_{k,\ell}  \left( \mu_j^{1/2} V^{(j)*}_k A V_\ell^{(j)}  +   \mu_j^{-1/2}  V^{(j)}_\ell A V^{(j)^*}_k  \right)\right)\ .
\end{equation}
\end{thm}

\begin{proof}  Let us consider the case {\it (1)}. For each $j =1,\dots,d$, apply the Gram-Schmidt algorithm to produce an invertible lower triangular $M_j\times M_j$ matrix $L^{(j)}$ and an orthonormal set $\{W^{(1)}_1,\dots, W^{(j)}_{M_j}\}$ such that
$$
V^{(j)}_k = \sum_{\ell =1}^{M_j} L^{(j)}_{k,\ell} W^{(j)}_\ell \qquad{\rm and\ hence}\qquad (V^{(j)})^*_k  = \sum_{\ell =1}^{M_j} \overline{L^{(j)}_{k,\ell} }(W^{(j)})^*_\ell\ .$$
Evidently $\{W^{(1)}_1,\dots, W^{(j)}_{M_j}\}$  lies in the eigenspace of $\Delta$ corresponding to $\mu_j$.  Then 
$$
\Phi(A) =   \sum_{j=1}^d\left( \sum_{k,\ell = 1}^{M_j}  \left( \mu_j^{1/2} [(L^{(j)})^*L^{(j)}]_{k,\ell} W^{(j)*}_k A W_\ell^{(j)}  +   \mu_j^{-1/2} \overline{ [(L^{(j)})^*L^{(j)}]_{k,\ell}} W^{(j)}_k A W^{(j)^*}_\ell  \right)\right)\ ,
$$
One can read off from this expression the characteristic matrix $C_\Phi$  with respect to the orthonormal basis obtained by extending, if necessary 
$$\{ W^{(j)}_k\ :\ 1 \leq j \leq d\ , 1 \leq k \leq M_j\ \}\ .$$
Let $C_\Psi$ denote the characteristic matrix of $\Psi$ with respect to this same basis. Then $\Phi - \Psi$ is CP if and only if $C_\Phi - C_\Psi \geq 0$. Thus, $\Phi - \Psi$, which is certainly self-adjoint on $\H_{\delta_s}$, is CP if and only if there are matrices $\{R^{(1)},\dots, R^{(d)}\}$
where for each $j$, $R^{(j)}$ is an $M_j\times M_j$ matrix with $0 \leq R^{(j)} \leq  (L^{(j)})^*L^{(j)}$ such that 
$$
\Psi(A) =   \sum_{j=1}^d\left( \sum_{k,\ell = 1}^{M_j}  \left( \mu_j^{1/2} R^{(j)}_{k,\ell} W^{(j)*}_k A W_\ell^{(j)}  +   \mu_j^{-1/2} \overline{ R^{(j)}_{k,\ell}} W^{(j)}_k A W^{(j)^*}_\ell  \right)\right)\ .
$$
Now defining $T^{(j)} = ((L^{(j)})^{-1})^* R^{(j)} (L^{(j)})^{-1}$, we have the result in case {\it (1)} on account of the self-adjointness of each $T^{(j)}$. 

The proof in case {\it (2)} is essentially the same, except for one point: Since $\{V^{(1)}_1,\dots, V^{(1)}_{M_1}\}$ is self-adjoint, for all $i,j$, $\tr[(V^{(1)}_i)^*V^{(1)}_j] = \tr[V^{(1)}_iV^{(1)}_j] \in \R$. 

Therefore, applying the Gram-Schmidt algorithm to 
$\{V^{(1)}_1,\dots, V^{(1)}_{M_1}\}$ yields a self-adjoint orthonormal basis $\{W^{(1)}_1,\dots, W^{(1)}_{M_1}\}$ and a {\em real} lower triangular matrix $L^{(1)}$ such that 
$V^{(1)}_k = \sum_{\ell =1}^{M_j} L^{(1)}_{k,\ell} W^{(1)}_\ell$.    Also, since $\{W^{(1)}_1,\dots, W^{(1)}_{M_1}\}$ is a set of  self-adjoint eigenvectors of $\Delta$ in the eigenspace with eigenvalue $1$,  
$\sum_{j=1}^{M_1} R^{(1)}_{i,j} W^{(1)}_i A W^{(1)}_j$ is CP and self-adjoint on $\H_{\delta_s}$ if and only if  $ R^{(1)} \geq 0$ and $R^{(1)}$ is real. Hence the matrix $T^{(1)}$ is {\em real} in addition to satisfying $0 \leq T^{(1)} \leq \one$.

\end{proof}

\begin{thm}\label{Main6}  Let $\Phi\in CP(\one)_{\delta_s}$, the set of unital CP maps that are self-adjoint on $\H_{\delta_s}$.   Let $\Phi$ have a canonical minimal Kraus representation specified in terms of 
$$
\{\ \{V^{(1)}_1,\dots, V^{(1)}_{M_1}\}, \dots, \{V^{(d)}_1,\dots, V^{(d)}_{M_d}\} \ \}
$$
as in \eqref{canform1} or \eqref{canform2}.
Define 
$$X^{(j)}_{k,\ell} = \mu_j^{1/2} V^{(j)*}_kV_\ell^{(j)}  +    \mu_j^{-1/2} V^{(j)}_\ell V^{(j)^*}_k \ .$$
The necessary and sufficient condition for $\Phi$ to be extremal in $CP(\one)_{\delta_s}$ is that for each $1\leq j \leq d$,
\begin{equation}\label{lininsets}
\{ \ X^{(j)}_{k,\ell}\ \ :\ 1\leq k,\ell\leq M_j \ , \ 1 \leq j \leq d\ \}
\end{equation}
is linearly independent.
\end{thm}

\begin{proof}  Suppose first that $\Phi$ has a canonical minimal
Kraus representation of the type \eqref{canform1}.
Let $\Psi$ be a unital CP map that is self-adjoint on $\H_{\delta_s}$, and suppose that for some $0< t < 1$, $\Phi - t\Psi$ is CP. Then by Theorem~\ref{Main4}, there are matrices $T^{(1)},\dots, T^{(d)} $ where $T^{(j)}$ is an $M_j\times M_j$ matrix satisfying $0 \leq T^{(j)} \leq \one$ such that 
$t\Psi(A)$ is given by the right side of \eqref{Main41}.  Since $\Psi$ and $\Phi$ are both unital, $t\Psi(\one) = t\Phi(1)$, and then 
$$
\sum_{j=1}^d\left( \sum_{k,\ell = 1}^{M_j}  (T^{(j)}_{k,\ell}- t\delta_{k,\ell}) \left(  \mu_j^{1/2}V^{(j)*}_kV_\ell^{(j)}  +   \mu_j^{-1/2}  V^{(j)}_\ell V^{(j)^*}_k  \right)\right) = 0\ .
$$
Then, if the set specified in \eqref{lininsets} is linearly independent, for each $j$, $T^{(j)} = t\one$, and hence $t\Psi = t \Phi$ so that $\Psi = \Phi$.  Hence $\Phi$ is extremal. 

For the converse, suppose that  the set specified in \eqref{lininsets} is linearly dependent. Then there are matrices $\{B^{(1)},\dots,B^{(d)}\}$, not all zero,  such that 
\begin{equation}\label{baskrid31}
\sum_{j=1}^d\left( \sum_{k,\ell = 1}^{M_j}  B^{(j)}_{k,\ell} \left(  \mu_j^{1/2}V^{(j)*}_kV_\ell^{(j)}  +   \mu_j^{-1/2}  V^{(j)}_\ell V^{(j)^*}_k  \right)\right) = 0\ .
\end{equation}
The adjoint of the $N\times N$ matrix on the left in \eqref{baskrid31}  equals the matrix obtained by replacing each $B^{(j)}$ by its adjoint. Thus, we may assume without loss of generality that each $B^{(j)}$ is self-adjoint.

Now replacing each $B^{(j)}$ with $tB^{(j)}$ for some common $t>0$, we may assume without loss of generality that $\|B^{(j)}\| \leq 1$ for each $j$. Now define $T^{(j)} := \tfrac12(\one + B^{(j)})$. Then  $0 \leq T^{(j)} \leq \one$ for all $j$. 
Now, using these $T^{(j)}$,  define $\Psi$ by \eqref{Main41}.  Then by Theorem~\ref{Main4}, $\Psi$ is CP, self-adjoint on $\H_{\delta_s}$ and $\Phi - \Psi$ is CP.   By \eqref{baskrid31} and the fact that  $\Phi$ is unital, $\Psi$ is unital. But since $B^{(j)} \neq 0$ for at least one $j$, $\Psi$ is not a multiple of $\Phi$. Hence $\Phi$ is not extremal. 

The case in which $\Phi$ has a canonical minimal
Kraus representation of the type \eqref{canform2} is quite similar. 
 \end{proof}

\section{evenly self-adjoint maps}

We say a map $\Phi$ is {\em evenly self-adjoint} in case it is self-adjoint on $\H_m$ for all even $m$.  Let $CP_{even}$, $CP_{even}(\one)$ and $QMS_{even}$ be the sets of evenly symmetric CP maps, unital CP maps  and QMS generators respectively.
We have seen that, for instance, $CP_{GNS} \subset CP_{even} \subset CP_{KMS}$, since when $\Phi\in CP_{GNS}$, $\Phi$ is self-adjoint on {\em every} $\H_m$ whether $m$ is even or not, and since the measure $m$ defining the KMS inner product is even. For the same reason we have $CP_{even} \subset CP_{BKM}$. 

Using the next lemma, we will describe a natural way to construct elements of $CP_{even}$, $CP_{even}(\one)$ and $QMS_{even}$ that do not belong to $CP_{GNS}$, $CP_{GNS }(\one)$ and $QMS_{GNS}$ respectively:

\begin{lm}\label{colm} Let $m\in\mathcal{P}[0,1]$ and suppose $\Phi$ is such that  $[\cM_m,\Phi] = 0$. Then $\tfrac12(\Phi+ \Phi^\dagger)$ and  $ \tfrac{1}{2i}(\Phi - \Phi^\dagger)$ are self-adjoint on $\H_m$.
\end{lm}

\begin{proof}  By Lemma~\ref{adlem}, $\Phi$ is self-adjoint on $\H_m$ if and only if $\cM_m\circ \Phi = \Phi^\dagger \circ \cM_m$.   Since $\cM_m^\dagger = \cM_m$,  $[\cM_m,\Phi^\dagger] = 0$.  Then
$$\cM_m\circ \tfrac12(\Phi^\dagger+\Phi) =      \tfrac12(\Phi^\dagger+\Phi) \circ \cM_m  =   \tfrac12(\Phi^\dagger+\Phi)^\dagger  \circ \cM_m \ .$$
The proof for $\tfrac{1}{2i} (\Phi - \Phi^\dagger)$ is the same. 
\end{proof}

For $1\leq i,j \leq N$, let $E_{i,j} := \sqrt{N}|u_i\rangle \langle u_j|$ where $\{u_1,\dots,u_N\}$ is an orthonormal basis of $\C^N$ consisting of eigenvectors of $\sigma$; $\sigma u_j = \lambda_j u_j$.  Recall that the $E_{i,j}$ are an orthonormal basis of $\mathfrak{H}$ consisting of eigenvectors of $\cM_m$ with $\cM_mE_{i,j} = (\lambda_i,\lambda_j)_m E_{i,j}$.  If $m$ is not even, it can easily be that each eigenspace of $\cM_m$ is one dimensional, and then $\Phi$ commutes with $\cM_m$ if and only if it is a function of $\cM_m$ itself. 

However, when $m$ is even,  $(\lambda_i,\lambda_j)_m = (\lambda_j,\lambda_i)_m$ for all $i,j$, and hence if $i < j$,  $E_{i,j}$ and $E_{j,i}$ belong to the same eigenspace. For $i < j$, consider the map $\Phi$ defined by  $\Phi(A) = E_{i,j} A E_{i,j}$.
Note that $\Phi^\dagger(A) = E_{j,i}A E_{j,i}$.
A simple calculation shows that $\Phi\circ \cM_m = (\lambda_i,\lambda_j)_m \Phi $ and $\cM_m\circ \Phi = (\lambda_j,\lambda_i)_m \Phi$.    Hence, whenever $m$ is even, $[\Phi, \cM_m] = 0$.   For $i\neq j$ define the maps  $\Psi_{i,j}$ by
\begin{equation}\label{evenly}
\Psi_{i,j}(A) := \begin{cases} \tfrac12(E_{i,j} A E_{i,j} + E_{j,i}A E_{j,i}) & i < j\\
 \tfrac{1}{2i}(E_{i,j} A E_{i,j} - E_{j,i}A E_{j,i}) & i > j\ . \end{cases}
\end{equation}
Lemma~\ref{colm} says that these maps are evenly self-adjoint.   

Therefore, let $\Phi_0$ be a CP map that is self-adjoint on $\H_m$ for all $m\in\mathcal{P}[0,1]$, and let $T$ be a real $N\times N$  matrix that is zero on the diagonal. Then 
\begin{equation}\label{evenly2}
\Phi = \Phi_0 + \sum_{i \neq j} T_{i,j} \Psi_{i,j}  
\end{equation}
is evenly self-adjoint, and is CP if and only if 
$C_{\Phi_0 +  \sum_{i \neq  j} T_{i,j}\Psi_{i,j} } \geq 0$.   Notice also that  $\Psi_{i,j}(\one)  = 0$ for all $i\neq j$,
so that, if $\Phi_0$ is unital, then so is the operator $\Phi$ in \eqref{evenly2}.  Likewise, if $\cL_0$ is a QMS generator that is self-adjoint on $\H_m$ for all $m\in\mathcal{P}[0,1]$, then 
\begin{equation}\label{evenly3}
\cL = \cL_0 + \sum_{i \neq j} T_{i,j} \Psi_{i,j} 
\end{equation}
is evenly self-adjoint, and is a QMS generator if and only if $R_{\cL_0 +  \sum_{i \neq j} T_{i,j}\Psi_{i,j} } \geq 0$ where the reduced density matrix is computed with respect to any unital basis. 

  As we explain next, under a non-degeneracy condition on the spectrum of the modular operator, this construction not only gives us a class of examples, but   a complete parameterization of the set of all  evenly self-adjoint  CP maps and QMS generators.

Suppose the eigenvalues $\{\lambda_1,\dots,\lambda_N\}$ of $\sigma$ are such that the %$\binom{N}{2}$ 
$N^2-N$ numbers $\frac{\lambda_i}{\lambda_j}$, $i \neq j$ are all distinct, which of course implies that the $N$ eigenvalues of $\sigma$ are all distinct. 
In this case we say that the modular operator has {\em minimally degenerate spectrum} -- the eigenvalue $1$ has multiplicity $N$ and all other eigenvalues are simple. 

\begin{thm}\label{evenlythm1} Suppose $\sigma$ is such that $\Delta$ has minimally degenerate spectrum.  Let $\Phi$ be an evenly self-adjoint CP map. Then there exists a GNS self-adjoint CP map $\Phi_0$ and a real $N\times N$ matrix $T$ that is zero on the diagonal such that $\Phi$ is given by \eqref{evenly2}. If we also assume $\Phi\in CP(\one)$, then $\Phi_0\in CP(\one)$.

Furthermore,
the extreme points of the set $CP_{even}$ of evenly self-adjoint CP maps are of the form either
\begin{equation}\label{evenlyex1}
\Phi(A) =  VAV 
\end{equation}
where $\Delta(V) = V$ and  $V^* = V$, in which case $\Phi$ is GNS self-adjoint, or, for some $\alpha$ such that $\alpha_1 < \alpha_2$,
\begin{equation}\label{evenlyex2}
\Phi(A) =  a\left(  \lambda_{\alpha_1} E_{\alpha'} A E_\alpha  + \lambda_{\alpha_2} E_{\alpha} A E_{\alpha'} +  \sqrt{\lambda_{\alpha_1}\lambda_{\alpha_2}} (e^{i\theta}E_{\alpha'}AE_{\alpha'} + 
e^{-i\theta}  E_{\alpha}AE_{\alpha}) \right)
\end{equation}
where $a>0$ and  $\theta \in [0,2\pi)$, in which case $\Phi$ is not GNS self-adjoint.

\textcolor{red}{\textbf{(What about extrema for $CP(\one)$?)}}

\end{thm}

\begin{thm}\label{evenlythm2} Suppose $\sigma$ is such that $\Delta$ has minimally degenerate spectrum. Let $\cL$ be an evenly self-adjoint QMS generator. Then there exists a GNS self-adjoint CP map $\cL_0$ and a real $N\times N$ matrix $T$ that is zero on the diagonal such that $\cL$ is given by \eqref{evenly3}. Furthermore,
the extreme points of the set $QMS_{even}$ of evenly self-adjoint QMS generators are of the form either
\begin{equation}\label{evenlyex1L}
\cL(A) =  VAV - \tfrac12(V^2 A + A V^2) 
\end{equation}
where $\Delta(V) = V$ and  $V^* = V$, in which case $\cL$ is GNS self-adjoint, or, for some $\alpha$ such that $\alpha_1 < \alpha_2$, 
\begin{eqnarray}\label{evenlyex2L}
\cL(A) &=&  a\left(  \lambda_{\alpha_1} E_{\alpha'} A E_\alpha  + \lambda_{\alpha_2} E_{\alpha} A E_{\alpha'} +  \sqrt{\lambda_{\alpha_1}\lambda_{\alpha_2}} (e^{i\theta}E_{\alpha'}AE_{\alpha'} + 
e^{-i\theta}  E_{\alpha}AE_{\alpha}) \right)\\
&-& \tfrac{a\sqrt{N} }{2}\bigg(  (\lambda_{\alpha_1} E_{\alpha_2,\alpha_2} +  \lambda_{\alpha_2} E_{\alpha_1,\alpha_1})A + A  (\lambda_{\alpha_1} E_{\alpha_2,\alpha_2} +  \lambda_{\alpha_2} E_{\alpha_1,\alpha_1})  \bigg)\ .
\end{eqnarray}
where $a>0$ and  $\theta \in [0,2\pi)$,  in which case $\cL$ is not GNS self-adjoint. 
\end{thm}

The proofs of these theorems are very similar. We first record some useful lemmas.

\begin{lm}\label{ratiolm} Let $a,b,c,d >0$ with $a\leq b$ and $c\leq d$. Then 
\begin{equation}\label{evenly7}
 \frac{(a,b)_m}{(c,d)_m}
 \end{equation} is independent of the even measure $m$ if and only if $\frac{a}{c} = \frac{b}{d}$. 
\end{lm}

\begin{proof}
Suppose $\frac{a}{c} = \frac{b}{d} = K$.  Then $a = Kc$ and $b= Kd$ so that for any $m$, 
$$\int_0^1a^s b^{1-s}{\rm d}m = \int_0^1(Kc)^s (Kd)^{1-s}{\rm d}m = K\int_0^1 c^s d^{1-s}{\rm d}m\ ,$$
and the ratio in \eqref{evenly7} is independent of $m$, even or not. 

For the converse, suppose that the ratio in \eqref{evenly7} is $K$ for all even $m$. Then taking $m =\delta_{1/2}$ and $m =\tfrac12(\delta_0 + \delta_1)$, giving the geometric and arithmetic means respectively, we have 
$$a+b = K(c+d) \quad {\rm and}\quad \sqrt{ab} = K \sqrt{cd}\ ,$$
from which we deduce
$$
(\sqrt{a} + \sqrt{b})^2 = K(\sqrt{c} + \sqrt{d})^2  \quad {\rm and}\quad (\sqrt{a} - \sqrt{b})^2 = K(\sqrt{c} - \sqrt{d})^2
$$
Since $a\leq b$ and $c\leq d$, 
$$\sqrt{a}  = \frac{\sqrt{a} + \sqrt{b}}{2} +  \frac{\sqrt{a} - \sqrt{b}}{2}  = \sqrt{K} \frac{\sqrt{c} + \sqrt{d}}{2} + \sqrt{K} \frac{\sqrt{c} - \sqrt{d}}{2}  =  \sqrt{K} \sqrt{c}\ ,$$
and likewise, $\sqrt{b} =  \sqrt{K}  \sqrt{d}$. \textcolor{red}{This implies that $\frac{a}{c}=\frac{b}{d}=K$.}
\iffalse
Thus $(\sqrt{a},\sqrt{b}) =  \sqrt{K} (\sqrt{c},\sqrt{d})$, and hence $(a,b) = K(c,d)$.
\fi
\end{proof}

\begin{lm}\label{zerolm} Let $\Phi$ be an evenly symmetric map. Suppose $1 \leq i,j,k,\ell \leq N$ are such that  
 there exist even $m_1$ and $m_2$  for which 
\begin{equation}\label{evenly4}
 \frac{(\lambda_i,\lambda_k)_{m_1}} {(\lambda_j,\lambda_\ell)_{m_1}} \neq  \frac{(\lambda_i,\lambda_k)_{m_2}} {(\lambda_j,\lambda_\ell)_{m_2}}
\end{equation}
Let $C_\Phi$ be the characteristic matrix of $\cL$ computed with respect to  $\{E_{i,j}\}$.    Then $(c_\Phi)_{(i,j),(k,\ell)}= 0$.
\end{lm}

\begin{proof}
By Theorem~\ref{main1}, for $\Phi$ to be self-adjoint on $\H_{m_1}$,
$$
(c_\Phi)_{(i,j),(k,\ell)}  = \frac{(\lambda_i,\lambda_k)_{m_1}} {(\lambda_j,\lambda_\ell)_{m_1}} (c_\Phi)_{(\ell,k),(j,i)} \ ,
$$
while for $\Phi$ to be self-adjoint on $\H_{m_2}$, the same relation must hold with $m_1$ replaced by $m_2$. By \eqref{evenly4}, this means that $(c_\Phi)_{(i,j),(k,\ell)}= 0$.
\end{proof} 

\begin{proof}[Proof of Theorem~\ref{evenlythm1}] Returning to the notation $\alpha = (\alpha_1,\alpha_2)$ and $\alpha' = (\alpha_2,\alpha_1)$,  by Lemma~\ref{zerolm},
$(c_\Phi)_{\alpha,\beta}= 0$ unless $\frac{(\lambda_{\alpha_1},\lambda_{\beta_1})_m}{(\lambda_{\alpha_2},\lambda_{\beta_2})_m}$ is independent of $m$ even.  By Lemma~\ref{ratiolm}, this means that either
$$
\frac{\lambda_{\alpha_1}}{\lambda_{\alpha_2}}   =  \frac{\lambda_{\beta_1}}{\lambda_{\beta_2}}  \quad{\rm or}\quad \frac{\lambda_{\alpha_1}}{\lambda_{\beta_2}}   =  \frac{\lambda_{\beta_1}}{\lambda_{\alpha_2}} \ .
$$
Since the spectrum of $\Delta$ is minimally degenerate, the first of these conditions is satisfied if and only if  either ($\alpha = \beta$) or ($\alpha = \alpha'$ and $\beta = \beta'$). The second of these conditions is satisfied if and only if
$\alpha = \beta$ or $\alpha = \beta'$. Thus, when $\Delta$ has minimally degenerate spectrum and $\Phi$ is evenly self-adjoint, then $(c_\Phi)_{\alpha,\beta} = 0$ unless one of the following is satisfied:

\smallskip
\noindent{\it (1)} $\alpha = \alpha'$ and $\beta = \beta'$

\smallskip
\noindent{\it (2)} $\alpha = \beta$

\smallskip
\noindent{\it (3)} $\alpha = \beta'$

If we order the indices so that $(1,1),\dots ,(N,N)$ come first, followed by consecutive pairs $(i,j)$ and $(j,i)$ with $i< j$, $C_\Phi$ will have an $N\times N$ block in the upper left, and then a string of $\binom{N}{2}$ $2\times 2$ blocks down the diagonal, with all other entries being zero.   Consider one of these $2\times 2$ blocks. The diagonal entries are $(c_\Phi)_{\alpha,\alpha}$  and $(c_\Phi)_{\alpha',\alpha'}$ for some $\alpha$ with $\alpha_1<  \alpha_2$.   Again by  Theorem~\ref{main1}, these are related by 
$$
(c_\Phi)_{\alpha,\alpha}  = \frac{\lambda_{\alpha_1}}{\lambda_{\alpha_2}}(c_\Phi)_{\alpha',\alpha'} 
$$
Hence if we order the indices so that $\alpha$ comes before $\alpha'$, the $2\times 2$ block has the form
$$a\left[\begin{array}{cc}  \lambda_{\alpha_1}  & \zeta\\ \overline{\zeta} & \lambda_{\alpha_2}\end{array}\right]$$
for some $a\geq 0$. Then if $a\neq 0$, we must have $|\zeta|^2 \leq \lambda_{\alpha_1}\lambda_{\alpha_2}$ if $\Phi$ is completely positive. If we write $\zeta = x+iy$, $x,y\in \R$, then 
\begin{equation}\label{2x2block}
a\left[\begin{array}{cc}  \lambda_{\alpha_1}  & \zeta\\ \overline{\zeta} & \lambda_{\alpha_2}\end{array}\right] =   a\left[\begin{array}{cc}  \lambda_{\alpha_1}  & 0\\ 0 & \lambda_{\alpha_2}\end{array}\right]  + x \Psi_{\alpha_1,\alpha_2} + y \Psi_{\alpha_2,\alpha_1} \ .
\end{equation}
Applying this to all such blocks we see that $\Phi$ has the form 
 \begin{equation}\label{evenly9}
\Phi = \Phi_0 + \sum_{i \neq j} T_{i,j} \Psi_{i,j}  
\end{equation}
where $C_{\Phi_0}$ is the matrix obtained  by setting all off diagonal elements of $C_\Phi$ outside the upper left $N\times N$ block equal to zero. Since $C_\Phi$ is positive semidefinite, so is $C_{\Phi_0}$. It follows that $\Phi_0$ is CP and GNS symmetric. Also note that $\Phi(\one)=\one$ is equivalent to $\Phi_0(\one)=\one$, since the maps $\Psi_{i,j}$ annihilate $\one$.

Maps of the form ~\eqref{evenlyex1} are readily seen to be extreme points of $CP_{even}$, just as in the proof of Theorem ~\ref{Main3}. Other extrema $\Phi$ are obtained by letting the only non-zero entries in $C_\Phi$ be those of a $2\times 2$ block like ~\eqref{2x2block} with $|\zeta|$ chosen as to give equality in the condition $|\zeta|^2 \leq \lambda_{\alpha_1}\lambda_{\alpha_2}$. These extreme points are those of the form ~\eqref{evenlyex2}. In particular note that maps of the form ~\eqref{SmsCP}, which were extreme in $CP_{GNS}$, are not extreme in $CP_{even}$, since they can be obtained as convex combinations of two maps in the form ~\eqref{evenlyex2} using opposite values of $\theta$.
\end{proof}

\begin{proof}[Proof of Theorem~\ref{evenlythm2}] Just as in the proof of Theorem~\ref{evenlythm1} above, we conclude that $C_\cL$ for an evenly self-adjoint QMS generator $\cL$ has the block structure of an upper-left $N\times N$ block followed by $\binom{N}{2}$ $2\times2$ blocks down the diagonal in the form~\eqref{2x2block}, which justifies the formula
 \begin{equation}
\cL = \cL_0 + \sum_{i \neq j} T_{i,j} \Psi_{i,j}  
\end{equation}
with $\cL_0$ being GNS self-adjoint. It remain to prove that $\cL_0$ is a QMS generator.

According to Lemma~\ref{redLm}, the reduced characteristic matrix $R_\cL$ of $\cL$ computed in some unital orthonormal basis must be positive semidefinite. Moreover, as a consequence of Lemma~\ref{exg} and as described in Remark~\ref{remark_unital_basis}, $R_\cL$ can be obtained as the lower-right $(N^2-1)$-dimensional block of the matrix $\widetilde{C}_\cL = UC_\cL U^\ast$, where $U$ is an unitary matrix containing a nonzero upper-left $N\times N$ block and the identity for its lower-right $(N^2-N)\times(N^2-N)$ block. In particular $R_\cL$ has a block structure with an upper-left $(N-1)\times(N-1)$ block and the same $\binom{N}{2}$ $2\times2$ blocks down the diagonal as $C_{\cL}$, and by assumption each of these blocks is positive semidefinite, which implies the same condition $|\zeta|^2 \leq \lambda_{\alpha_1}\lambda_{\alpha_2}$ described in the proof of the previous theorem. Meanwhile, $R_{\cL_0}$ also has a block structure with the same upper-left $(N-1)\times(N-1)$ block as $R_\cL$ (because the upper-left $N\times N$ blocks for $C_\cL$ and $C_{\cL_0}$ are the same), and with a diagonal lower-right $(N^2-N)\times(N^2-N)$ block. Hence $R_{\cL_0}$ is positive semidefinite, implying that $\cL_0$ is a QMS generator.

The description of the extreme points of $QMS_{even}$ follows the same reasoning as in the proof of Theorem~\ref{evenlythm1} above.
\end{proof}

\section{From KMS self-adjointness to BKM self-adjointness}

 Recall that $\cM_{BKM}$ denotes the operator on $\fH$ given by
\begin{equation}\label{Mdef}
\cM_{BKM}(A) = \int_0^1 \sigma^{s} A \sigma^{1-s} {\rm d}s\ .
\end{equation}
Then
\begin{equation}\label{Minv}
\cM_{BKM}^{-1}(A) = \int_0^\infty \frac{1}{t+ \sigma } A \frac{1}{t+ \sigma }  {\rm d}t \ ,
\end{equation}
so that $\cM_{BKM}^{-1}$ is CP.  
Now define the unital CP map
\begin{equation}\label{transfor1}
\Psi(A) = \int_0^\infty \frac{\sqrt{\sigma}}{t+ \sigma } A \frac{\sqrt{\sigma}}{t+ \sigma }  {\rm d}t\ .
\end{equation}

\begin{thm} The map $\Phi \mapsto \Psi\circ \Phi$ is a one-to-one map from the set of KMS self-adjoint CP maps into the set of BKM self-adjoint CP maps, taking unital maps to unital maps.
\end{thm}

\begin{proof}  Let $\Phi$ be KMS self-adjoint. Then
\begin{eqnarray*}
\langle B, \Psi(\Phi(A))\rangle_{BKM} &=& \tr[B^* \cM_{BKM}(\Psi(\Phi(A)))]\  = \tr[B^* \sqrt{\sigma} \Phi(A)\sqrt{\sigma}] \\
&=& \langle B, \Phi (A) \rangle_{KMS} = \langle \Phi(B),A\rangle_{KMS}\\
&=&  \tr[\sqrt{\sigma} (\Phi(B))^* \sqrt{\sigma} A] =   \tr[\cM_{BMK}^{-1}(\sqrt{\sigma} (\Phi(B))^* \sqrt{\sigma}), \cM_{BKM}(A)]\\
&=&  \langle \Psi(\Phi(B)), A\rangle_{BKM}\ .
\end{eqnarray*}
The rest follows from the fact that $\Psi$ is invertible, CP, and unital. 

\end{proof}

Evidently, a similar construction is possible for any $m$ such that $\cM_m^{-1}$ is CP. This is not the case for every even $m$, but there are examples other than $BKM$ and $KMS$:

\begin{thm}  For $s\in [0,1]$ define $m_s := \frac12(\delta_s + \delta_{1-s})$. Then for each $s\in [0,1]$,  the operator $\cM_{m_s}^{-1}$ is completely positive for all $\sigma\in \Dens_+$. 
\end{thm}

\begin{proof} By Lemma~\ref{charMM}, for any $m$, for the standard matrix unit basis $\{E_\alpha\}$,
$$
(c_{\mathcal{M}^{-1}_m})_{\alpha,\beta} = (\alpha_1,\beta_1)^{-1}_m\delta_{\alpha_1,\alpha_2}\delta_{\beta_1,\beta_2}\ .
$$
If we order the indices as usual so that $(1,1),\dots,(N,N)$ come first, $C_{\cM_m}$ has the $N\times N$ matrix 
$\Lambda^{(m)}$ defined by 
\begin{equation}\label{lammdef}
\Lambda^{(m)}_{i,j} :=  \frac{1}{(\lambda_i,\lambda_j)_{m}}
\end{equation}
in its upper left block, and is zero elsewhere.  By Lemma~\ref{CPchar} $\cM_m^{-1}$ is CP if and only if  $\Lambda^{(m)}$ is positive semi-definite. Specializing to the case $m = m_s$ for some $s$,
 $$
 \Lambda^{(m_s)}_{i,j} =   \lambda_i^{-s} \lambda_j^{-s}  \frac{2}{\lambda_i^{1-2s} +  \lambda_j^{1-2s}} \ .
 $$

Define $\kappa_j = \lambda_j^{1-2s}$.   Then
$$
\frac{1}{\lambda_i^{1-2s} +  \lambda_j^{1-2s}} = \int_0^1  t^{\kappa_i+\kappa_j -1}{\rm d}t\ ,
$$
and hence for any $(z_1,\dots,z_n)\in \C^n$, 
$$\sum_{i,j=1}^n z_i^* \Lambda^{(m_s)}_{i,j} z_j = 2\int_0^1 \left| \sum_{i=1}^n \lambda_i^{-s} t^{\kappa_i-\frac12} z_i\right|^2 {\rm d}t \geq 0\ .$$
proving that $\Lambda^{(m_s)}$ is positive semidefinite.
\end{proof}

\section{$QMS_m$ for $N=2$,  $m$ even}

The following lemma will facilitate the computations in this section:

\begin{lm}\label{SACON} 
Let $\Phi$ be self-adjoint on $\H_m$. Then, with respect to any matrix unit orthonormal basis of $\mathfrak{H}$, the characteristic matrix $C_\Phi$ of $\Phi$ satisfies
\begin{equation}\label{case1}
{\rm For}\  \alpha = \alpha',\beta = \beta'\ ,\quad (c_\Phi)_{\alpha,\beta} = (c_\Phi)_{\beta,\alpha}\ ,
\end{equation}
\begin{equation}\label{case2}
{\rm For}\  \alpha = \alpha',\beta \neq \beta'\ ,\quad (c_\Phi)_{\alpha,\beta}(\lambda_{\alpha_2},\lambda_{\beta_2})_m = (c_\Phi)_{\beta',\alpha}(\lambda_{\alpha_2},\lambda_{\beta_1})_m\ ,
\end{equation}
\begin{equation}\label{case3}
{\rm For}\ \alpha \neq \alpha'\ ,\ \alpha =  \beta'\ ,\quad (c_\Phi)_{\alpha,\beta}(\lambda_{\alpha_2},\lambda_{\alpha_1})_m = (c_\Phi)_{\alpha,\beta}(\lambda_{\alpha_1},\lambda_{\alpha_2})_m\ ,
\end{equation}
\begin{equation}\label{case4}
{\rm For}\ \alpha \neq \alpha'\ ,\ \alpha =  \beta\ ,\quad (c_\Phi)_{\alpha,\beta}\lambda_{\alpha_2} = (c_\Phi)_{\alpha,\beta}\lambda_{\alpha_1}\ ,
\end{equation}
\end{lm}

\begin{proof} This is an immediate consequence of Theorem~\ref{main1}.
\end{proof}

In this section we determine the structure of $QMS_m$ for even $m$ when $N=2$. 
First consider any Hermitian $\cL$ on $\cM_2(\C)$  such that $\cL(\one) = 0$.  
Then the fact that 
$C_\cL$ is self-adjoint, together with Lemma~\ref{pairlem} constrain $C_\cL$ to have the form
  \begin{equation}\label{firstred}
  C_\cL = \left[\begin{array}{cccc} 
 -a &  \phantom{-}\zeta_1 & z & \zeta_2\\
 \phantom{-}\overline{\zeta_1} & -b &\zeta_4 & -\overline{z}\\
 \phantom{-}\overline{z} & \phantom{-}\overline{\zeta_4} & b &\zeta_3\\
 \phantom{-}\overline{\zeta_2} & -z &\overline{\zeta_3} & a\end{array}\right]\ ,
  \end{equation}
where $a,b$ are real and $z$ is complex and the relations among entries involving them are determined by Lemma~\ref{pairlem} and self-adjointness, 
and where $\zeta_1,\dots,\zeta_4$ are complex, and the relations among entries involving them are constrained  only by self-adjointness.

We next apply Lemma~\ref{SACON}. 
By \eqref{case1}, $\zeta_1 = \zeta_1^* =: -x$, and hence the upper-left block is real and symmetric for all choices of $m$. By \eqref{case3}, since $(\lambda_1,\lambda_2)_m = (\lambda_2,\lambda_1)_m$ there is no restriction on $\zeta_3$. 
 
 Next we apply \eqref{case2}.
 Taking $\alpha= (1,1)$ and $\beta = (1,2)$, and then  $\alpha= (2,2)$ and $\beta = (2,1)$, we see that  
 $$z  = \frac{\lambda_1}{(\lambda_1,\lambda_2)_m}\overline{\zeta_2} \quad{\rm and}\quad   
-z  = \frac{ \lambda_2 } { (\lambda_2,\lambda_1)_m }\zeta_4\ .$$

Next we apply \eqref{case4}. Taking $\alpha= (1,2)$ and $\beta = (1,2)$
$$ b =   \frac{\lambda_1}{\lambda_2} a\ .$$

We conclude that if we set $\mu_j := (\lambda_1,\lambda_2)_m/\lambda_j$, $j=1,2$, and replace $a$ by $\lambda_2 a$, 

 \begin{equation}\label{firstred2}
  C_\cL = \left[\begin{array}{cccc} 
 -a &  -x & z & \mu_1 \overline{z}\\
 -x & -b &-\mu_2 z & -\overline{z}\\
 \phantom{-}\overline{z} & -\mu_2 \overline{z} & b &\zeta_3\\
 \phantom{-}\mu_1 z & -z &\overline{\zeta_3} & a\end{array}\right]\ ,
  \end{equation}

  Conjugating with the unitary $U$ given in \eqref{passtounit2}, and replacing $z$ by $\sqrt{2}z$, yields $\widetilde{C}_\cL$, the characteristic matrix for the associated unital basis:
 \begin{equation}\label{firstred22}
 \widetilde{C}_\cL  =  \left[\begin{array}{cccc} 
 -x- \frac{a}{2} & (\lambda_1-\lambda_2)\frac{a}{2} & (1-\mu_2)z & (\mu_1-1)\overline{z}\\
(\lambda_1-\lambda_2)\frac{a}{2}  & x- \frac{a}{2} & (1+\mu_2)z & (1+\mu_1)\overline{z}\\
(1-\mu_2)\overline{z} &  (1+\mu_2)\overline{z} & \lambda_1 a & \zeta_3\\
(\mu_1-1)z & (1+\mu_1)z & \overline{\zeta_3} & \lambda_2 a\end{array}\right]\ .
 \end{equation}

Defining $\nu_j = \mu_j+1$, $j=1,2$ and dropping the subscript on $\zeta_3$, we have the reduced characteristic matrix
 \begin{equation}\label{firstred222}
\widetilde{R}_\cL  = 
   \left[\begin{array}{ccc} 
 x-  \frac{a}{2} &\nu_2 z & \nu_1 \overline{z}\\
 \nu_2 \overline{z} & \lambda_1 a &\zeta\\
\nu_1 z &\overline{\zeta} & \lambda_2 a\end{array}\right]
 \end{equation}

To break the homogeneity, let us fix  $\tr[R_{\cL}] =1$, which means $x+ \frac{a}{2}=1$.  For positivity, we must have $x\geq \frac{a}{2} >0$, and hence $0 \leq 1\leq a$. 
  We now determine the extreme points in the set of $(a,z,\zeta)$ for which $\widetilde{R}_\cL$ is positive semidefinite. 

If $a=0$, then necessarily $z=\zeta  =0$, so $(a,z,\zeta) = (0,0,0)$ is extreme. If $a=1$, then $x-\frac{a}{2} =0$, and  necessarily $z=0$. 
Then extremality reduces to $|\zeta| = \sqrt{\lambda_1\lambda_2}$.   Thus,
$$(1,0, \sqrt{\lambda_1\lambda_2}e^{i\theta})$$
yields a one parameter family of extreme points.  

Now we turn to the cases in which $0 < a < 1$.  Suppose first that 
\begin{equation}\label{exx5}
|\zeta| =  a\sqrt{\lambda_1\lambda_2}\ .
\end{equation}
As with $a=1$, it is again the case that $z=0$ is necessary for positivity of $R_\cL$. To see this, define
$$
A :=  \left[ \begin{array}{cc} \lambda_1a & \zeta\\  \overline{\zeta}  & \lambda_2a\end{array}\right]
 \ ,\quad 
\vec v :=(z\nu_2,\overline{z}\nu_1)\ ,\quad 
 \vec{\eta} := (\eta_1,\eta_2)\quad{\rm and}\quad 
\vec{Z} := (t,\eta_1,\eta_2)
$$
where $\eta_1,\eta_2\in \C$ and $t\in \R$.  Then, again with $x=1$, 
$$\langle \vec{Z}, R_{\cL} \vec{Z}\rangle = t^2(1-a) + 2t \Re(\langle \vec{v},\vec{\eta}\rangle) + \langle \vec{\eta},A\vec{\eta}\rangle\ .$$
Minimizing over $t$, we find $\langle \vec{\eta},B(z)\vec{\eta}\rangle$ where
$$
B(z) := \left[\begin{array}{cc}  
a\lambda_1 - \frac{|z|^2}{(1-a)}\nu_2^2 & \zeta -  \frac{\overline{z}^2}{(1-a)}\nu_1\nu_2 \\
\overline{\zeta} -  \frac{z^2}{(1-a)}\nu_1\nu_2 & a\lambda_2 -  \frac{|z|^2}{(1-a)}\nu_1^2
\end{array}\right]\ .
$$
When \eqref{exx5} is satisfied, $A$ has rank one, and $B(z)$ is the difference of two positive rank one matrices. Hence $B(z)$ cannot be positive definite unless the two rank one matrices are proportional.
Since the null space of $A$ is spanned by  $w := (-\sqrt{\lambda_2},\sqrt{\lambda_1}e^{i\theta})$ where $\zeta = |\zeta |e^{i\theta}$,  when $y\neq 0$, it is only possible for $B(z)$ to be positive in case 
$\langle v,w\rangle = 0$. However,  $|\langle v,w\rangle| \geq |z||\sqrt{\lambda_1}\nu_1-\sqrt{\lambda_2}\nu_2 |$, and
$$
|\sqrt{\lambda_1}\nu_1 - \sqrt{\lambda_2}\nu_2| = \left(\frac{(\lambda_1,\lambda_2)_m}{\sqrt{\lambda_1\lambda_2}} - 1\right)|\sqrt{\lambda_2} - \sqrt{\lambda_1}|\ .
$$
Under the assumption $\lambda_1\neq \lambda_2$, this is non-zero unless $m$ is the point mass at $1/2$; i.e., in the KMS case.

Thus,  when \eqref{exx5} is satisfied, except in the KMS case, $z=0$ is necessary and sufficient for the reduced characteristic matrix to be positive semidefinite. 
We set aside the KMS case since we already have a complete description of it for general $N$. 
 Then
\begin{equation}\label{extremep1}
(0,0,0) \qquad{\rm and}\quad (1,0, \sqrt{\lambda_1\lambda_2}e^{i\theta}) \ , \quad 0 \leq \theta < 2\pi\ ,
\end{equation}
are all of the extreme points where \eqref{exx5} is satisfied. Notice that they are exactly the ones that are evenly symmetric.

Now consider pairs $(a,\zeta)$ such that $|\zeta| < a\sqrt{\lambda_1\lambda_2}$. 
Then $B(0) > 0$.  Writing $z =: |z|e^{i\varphi}$ and holding $\phi$ fixed,  there is $r_0>0$ such that for  $|z| \leq r_0$, $B(|z|e^{i\phi}) \geq 0$, but for $|z| > r_0$, $B(z)$ has a negative eigenvalue. It follows that a necessary condition for extremality is that $\det(B(|z|e^{i\phi}) ) =0$.
Computing $\det(B(|z|e^{i\phi}))=0$ yields
\begin{equation}\label{exx6}
r_0^2 = \frac{ (1-a)(a^2\lambda_1\lambda_2   -|\zeta|^2) }{a(\lambda_1\nu_1^2+\lambda_2\nu_2^2)  -2\nu_1\nu_2\Re (\zeta e^{i2\phi}) } \ .
\end{equation}
Note that the denominator in this expression is strictly positive since
$|\Re (\zeta e^{i2\phi}) | \leq a \sqrt{\lambda_1\lambda_2}$ and hence
\begin{eqnarray*}\left(a(\nu_2^2\lambda_2 + \nu_1^2\lambda_1) - \nu_1\nu_22\Re (\zeta e^{i2\phi}) \right) &\geq&  \left(a(\nu_2^2\lambda_2 + \nu_1^2\lambda_1) -2 \nu_1\nu_2a\sqrt{\lambda_1\lambda_2} \right)\\
& =&a(\nu_1\sqrt{\lambda_1} - \nu_2\sqrt{\lambda_2})^2 > 0\ .
\end{eqnarray*}

Thus the remaining extreme points are those given by 
\begin{equation}\label{extremep1}
(a, r_0e^{i\phi} , r\sqrt{\lambda_1\lambda_2}e^{i\theta}) \qquad 0< a <1\ , \quad 0 \leq r < a\ ,\quad 0 \leq \theta,\varphi < 2\pi\ ,
\end{equation}
with $r_0$ given by \eqref{exx6}. 

To write $\cL$ in the canonical form $\cL(A) = G^*A + AG + \Phi(A)$, we easily read off from \eqref{firstred22}:
$$
G  = \left[\begin{array}{cc} \tfrac{a}{4}(3\lambda_1 -\lambda_2) - \tfrac{x}{2}  &  \tfrac{z}{\sqrt{2}}(\mu_1-\mu_2)\\  \tfrac{\overline{z}}{\sqrt{2}}(\mu_1-\mu_2)  &  \tfrac{a}{4}(3\lambda_2 -\lambda_1) - \tfrac{x}{2} \end{array}\right]  +
\frac{1}{\sqrt{2}} (2 - \mu_1 - \mu_2) \left[\begin{array}{cc} \phantom{-}0 & z \\ -\overline{z} & 0\end{array}\right]\ .
$$
Thus writing $G = H+iK$, $H$ and $K$ self-adjoint, we have that 
$$
K := \frac{1}{\sqrt{2}i} (2 - \mu_1 - \mu_2) \left[\begin{array}{cc} \phantom{-}0 & z \\ -\overline{z} & 0\end{array}\right]
$$
and by \eqref{firstred2},
$$
z = \langle u_1, \cL(|u_1\rangle \langle u_1|)u_2\rangle\ .
$$
It is easy to check that in the KMS case, this formula for $K$ coincides with that given by Theorem~\ref{Main7A}.  To write $\Phi$ in Kraus form amounts to diagonalizing the $3 \times 3$ matrix $R_\cL$ given by \eqref{firstred222}. Evidently this can be done in closed form, but the resulting formulas are complicated and shed little light on matters.

\section{Appendix}

In this appendix we recall some ideas of Arveson that were originally developed in the context of minimal Stinespring representation, but which have, in our finite dimensional setting, a very simple expression in terms of minimal Kraus representations.  This complements previous work by Choi \cite{Choi75}. The required background on Stinespring representations can be found in the initial chapters of \cite{Paul02}.

Let $\Phi$ be a completely positive map on the algebra $\cA = {\mathcal M}_N(\C)$, and let
 ${\displaystyle 
 \Phi(A) = \sum_{j=1}^M V_j^* A V_j}$
 be a Kraus representation of it.   The same data can be cast as a Stinespring representation of $\Phi$, which was actually Kraus' starting point.
  Let $\mathfrak{H}_{N,M}$ denote the Hilbert space consisting of all $N\times M$ matrices $X$ equipped with the Hilbert-Schmidt inner product.  Define a representation of ${\mathcal M}_N(\C)$ on $\mathfrak{H}_{N,M}$ by
 \begin{equation}\label{pidef}
 \pi(A) X = AX\ .
 \end{equation}
 Define a linear transformation ${\mathcal V}: \C^N \to \mathfrak{H}_{N,M}$ by 
 \begin{equation}\label{Vdef}
 {\mathcal V} x= [V_1x,\dots,V_Mx]\ .
 \end{equation}
 where $[x_1,\dotsm x_M]$ denotes the $N\times M$ matrix whose $j$th column is $x_j$. Then for $X = [x_1,\dotsm , x_M]\in \mathfrak{H}_{N,M}$,
 ${\displaystyle \langle X,{\mathcal V} x\rangle_{\mathfrak{H}_{N,M}} = \tr[[x_1,\dotsm  , x_M]^*[V_1x,\dots,V_Mx]] = \sum_{j=1}^M \langle V_j^*x_j,x\rangle}$,
 from which it follows that
 \begin{equation}\label{stine1}
 {\mathcal V}^* X =  \sum_{j=1}^M V_j^* x_j\ .
 \end{equation}
 
 Thus we have the Stinespring representation 
  \begin{equation}\label{stine1}
  \Phi(A) = {\mathcal V}^* \pi(A){\mathcal V}\ ,
   \end{equation} 
   in terms of a map $\mathcal{V}$ from $\C^N$ into some other Hilbert space, and a representation of ${\mathcal M}_N(\C)$ on that Hilbert space. Stinespring's Theorem says that every CP map is of this form. 
  
  The  Stinespring representation \eqref{stine1} is {\em minimal} in case  the closed span of 
  $$\{ \pi(A){\mathcal V} x\ :\ A\in \cA\ ,\ x\in \C^N\}$$
  is all of $\mathfrak{H}_{N,M}$, the closure being irrelevant in this finite dimensional case. 
  
  \begin{lm} The Stinespring representation \eqref{stine1} specified by \eqref{Vdef} is minimal if and only if $\{V_1,\dots,V_M\}$ is linearly independent. 
  \end{lm}
  
 \begin{proof}
  Suppose that $W = [w_1,\dots,w_M]$ is non-zero and orthogonal to $\pi(A) {\mathcal V} x$ for all $A\in \cA$ and all $x\in \C^N$. Then
  $0 = \langle W,A{\mathcal V} x\rangle_{\mathfrak{H}_{N,M} } = \langle {\mathcal V}^*A^* W, x\rangle$, and hence the orthogonality is equivalent to the condition  ${\mathcal V}^*A^* W = 0$
 for all $A\in \cA$. 
 
 Suppose that $\{V_1,\dots, V_M\}$ is not linearly independent. Then there is a non-zero vector $v$ such that $\sum_{j=1}^M v_j V_j^* = 0$.  Let $x\in \C^N$ be arbitrary, and define $W = [v_1 x, \dots, v_M x]$. Then 
 $A^*W =  [v_1 A^*x, \dots, v_M A^*x]$ and 
 ${\displaystyle {\mathcal V}^*A^*W = \left(\sum_{j=1}^M v_j V_j^*\right)A^*x = 0}$.
 Thus, there exists a non-zero $N\times M$ matrix $W$ such that $W$ is orthogonal to $\pi(A){\mathcal V} x$ for all $A\in \cA$ and all $x\in \C^N$,
 which proves the necessity of the condition.

 Suppose $W$ is a non-zero  matrix such that  ${\mathcal V}^*A^*W = 0$ for all $A\in \cA$.  Let $r$ be the rank of $W$; evidently $0 < r \leq \min\{M,N\}$. Let $W = QR$ be a $QR$ factorization of $W$ so that $Q$ is an $N\times r$ matrix with orthonormal columns and $R$ is an $r\times M$  with linearly independent rows. 
  Write $Q = [q_1,\dots,q_r]$, and extend $\{q_1,\dots,q_r\}$  to an orthonormal basis $\{q_1,\dots,q_N\}$ of $\C^N$ if $r < N$. 
  Define an $N\times N$ matrix $B$ by $Bq_j = q_1$ for $j \leq r$, and, in case $r< N$, $Bq_j = 0$ for $j>r$. 
 Then $BQR = [q_1,\dots \ , q_1] R$; i.e., 
 ${\displaystyle(BQR)_{i,j} = (q_1)_i\sum_{k=1}^r  R_{k,j}}$.
 Since the rows of $R$ are linearly independent, the vector $v\in \C^M$ with $v_j = \sum_{k=1}^r  R_{k,j}$ is not zero. Hence   for any $w\in \C^N$, taking $A^* = UB$ where $U$ is an appropriate chosen unitary, we can arrange that  
 $A^*W = [v_1 w,\dots, v_M w]$, and then
 $$0 = {\mathcal V}^* A^*W = \left(\sum_{j=1}^M v_j  V^*_j \right)w\ ,$$
 and since $w$ is arbitrary, this implies that $\sum_{j=1}^M v_j  V^*_j = 0$.  This is impossible since $\{V_1,\dots,V_m\}$ are linearly independent, and hence $W$ must be zero. This proves the sufficiency of the condition. 
 \end{proof}
 
 The next theorems involve maps of the form
 \begin{equation}\label{Bform}
 \Psi(A) = \sum_{i,j=1}^M B_{i,j} V_i^* A V_j
 \end{equation}
 where $\{V_1,\dots,V_M\}\subset \cA$, and $B$ is an $M\times M$ matrix. Suppose that  $B$ is positive so that $B = S^*S$. Then writing $B_{i,j} = \sum_{k=1}^M S^*_{i,k}S_{k,j} =  \sum_{k=1}^M \overline{S_{k,i}}S_{k,j}$,  
 $$\Psi(A) = \sum_{k=1}^M W^*_k A W_k\quad{\rm where}\quad  W_k = \sum_{j=1}^M S_{k,j}V_j\ .
 $$
 It follows that whenever $B\geq 0$, $\Psi$ is CP, and this is true without any hypotheses on $\{V_1,\dots,V_M\}$.

 Now suppose that  $\{V_1,\dots,V_M\}$ is linearly independent. The Gram-Schmidt procedure yields an orthonomal set $\{E_1,\dots,E_M\}$ and an invertible lower triangular matrix $L$ such that  $V_i = \sum_{j=1}^M L_{i,j}E_j$.  
 Then 
 $$
\Psi(A) =  \sum_{i,j=1}^M B_{i,j} V_i^* A V_j  =  \sum_{i,j,k,\ell=1}^M B_{i,j} \overline{L_{i,k}}E_k^* A L_{j,\ell}E_\ell = \sum_{k,\ell=1}^M (L^*BL)_{i,j} E_k^* A E_\ell\ .
 $$
 Since $\{E_1,\dots,E_M\}$ is orthonormal, $\Psi$ is CP if and only if $L^*BL$ is positive semidefinite. But since $L$ is invertible, this is the case if and only if $B$ is positive semidefinite. Moreover, we see that there is at most one  matrix $B$ for which $\Psi$ can be written in the form \eqref{Bform} since $L^*BL$ is the characteristic matrix of $\Psi$ for an orthonormal basis determined by $\{V_1,\dots,V_M\}$. 
 This proves:
 
 \begin{lm}\label{cpchar2}
 Let $\Psi$ be a map defined by \eqref{Bform} for some set $\{V_1,\dots,V_M\}\subset {\mathcal M}_N(\C)$ and some $M\times M$ matrix $B$. Then if $B$ is positive semi-definite, $\Psi$ is CP, and if $\{V_1,\dots,V_M\}$ is linearly independent, $\Psi$ is CP if and only if $B$ is positive semi-definite, and in this case the corrrespondence between $\Psi$ and $B$ is one-to-one. 
 \end{lm}

 Now let $\Phi$ and $\Psi$ be two CP maps. Then $\Phi - \Psi$ is CP if and only if $C_\Phi - C_\Psi \geq 0$. Thus, the invertible transformation $\Phi \mapsto C_\Phi$ identifies the order structure on $CP(\cA)$ with the order structure on ${\mathcal M}_N(\C)^+$, the positive semidefinite elements of ${\mathcal M}_N(\C)$.  There is another characterization of this  order relation due to Arveson that has several advantages.
 
 \begin{thm}\label{arveRN}  Let $\Phi$ be a completely positive map given by a minimal Kraus representation $\Phi(A) = \sum_{j=1}^M V_j^* AV_j$.  Then a CP map $\Psi$ satisfies $\Phi \geq \Psi$ if and only if there is a uniquely determined  $M\times M$ matrix  $T$ such that $0 \leq T \leq \one$ and
 \begin{equation}\label{arv21}
 \Psi(A) = \sum_{i,j=1}^M T_{i,j} V_i^* A V_j\ . 
 \end{equation}
 Equivalently, in terms of the associated minimal Stinespring representation $\Phi(A) = {\mathcal V}^* \pi(A){\mathcal V}$,  there is a positive operator $\widetilde{T} \in \pi({\mathcal M}_N(\C))'$, the commutant of $\pi({\mathcal M}_N(\C))$, such that 
 \begin{equation}\label{arv22}
 \Psi(A) = {\mathcal V}^* \pi(A) \widetilde{T} {\mathcal V}\ . 
 \end{equation}
 \end{thm}
 
 \begin{remark} Arveson \cite[Theorem 1.4.2]{Arv69} proved the theorem in the second equivalent form, and discussed it as being a non-commutative Radon-Nikodym Theorem with  the Radon-Nikodym derivative being the element $\widetilde{T}$ of $\pi({\mathcal M}_N(\C))'$.

 \end{remark}
 
 \begin{proof}  Let $T$ be an $M\times M$ matrix, with $0 \leq T \leq \one$, and let $\Psi(A)$ be given by \eqref{arv21}.  Since $T\geq 0$, $\Psi$ is CP by Lemma~\ref{cpchar2}, and
 $$
 \Phi(A) - \Psi(A) = \sum_{i,j=1}^M (\delta_{i,j} -T_{i,j})V_i^* A V_j\ . 
 $$
 Since $\one - T \geq 0$,  by another application of Lemma~\ref{cpchar2}, $\Phi- \Psi $ is CP.

 Conversely, suppose that $\Phi - \Psi$ is CP. We again use the Gram-Schmidt procedure to produce  an orthonomal set $\{E_1,\dots,E_M\}$ and an invertible lower triangular matrix $T$ such that  $V_i = \sum_{j=1}^M L_{i,j}E_j$.  Then $\Phi(A) = \sum_{j=1}^M (L^*L)_{i,j}E_i^* A E_j$.
 
  If $M < N^2$, extend $\{E_1,\dots, E_M\}$ to an orthonormal basis of $\cM_N(\C)$ equipped with the Hilbert-Schmidt inner product. The characteristic matrix of $\Phi$ with respect to this basis,  $C_\Phi$,  has $L^*L$ as its upper-left $M\times M$ block, and all other entries are zero.   If $\Psi$ is a CP map such that $\Phi - \Psi$ is CP, then $C_\Phi - C_\Psi$ is positive semidefinite.  Hence for some $M\times M$ matrix $R$ with $0 \leq R \leq L^*L$,  
  $\Psi(A) = \sum_{i,j=1}^M R_{i,j} E_i^* A E_j$.  But then since  $E_j = \sum_{k=1}^M L^{-1}_{j,k}V_k$,
  $$
  \Psi(A) = \sum_{i,j=1}^M ((L^{-1})^*RL^{-1})_{i,j} E_i^* A E_j \ .
  $$
  Since $0 \leq R \leq L^*L$, $0 \leq (L^{-1})^*RL^{-1} \leq \one$.  Hence we may define $T := (L^{-1})^*RL^{-1}$ and we have the contraction.
   This completes the proof that a CP map $\Psi$ satisifes $\Psi \leq \Phi$ if and only if $\Psi$ has the form specified in \eqref{arv21}. 
 
 We now show that this is equivalent to $\Psi$ having the form specified in \eqref{arv22}.   For an $N\times M$ matrix $X$ and an $M\times M$ matrix $B$, define
 $$\pi'(B)X = XB^{\rm T}\ ,$$
where $B^{\rm T}$ denotes the transpose of $B$.  It is easy to check that $\pi'$ is a $*$-representation of $M_M(\C)$ on $\mathfrak{H}_{N,M}$. It is immediately clear that $\pi'(M_M(\C))$ lies in the commutant of $\pi({\mathcal M}_N(\C))$ since $\pi'$ acts by right multiplication, and $\pi$ acts by left multipication, and it is easy to check that in fact $(\pi(\cA))' =  \pi'(M_M(\C))$.

Now write $X\in \mathfrak{H}_{N,M}$ in the form $X = [x_1,\dots,x_M]$.  Then for any $M \times M$ matrix $R$
$$\sum_{j=1}^M \left[ R_{1,j}x_j, \dots, R_{M,j}x_j\right] = X R^{\rm T}\ .$$

Suppose now that  that $\Psi$ has the form specified in \eqref{arv21}. 
Then  for all $x,y\in \C^N$, 
 $$
 \langle x, \Psi (A)y\rangle = \sum_{i=1}^M \langle V_i x, A \sum_{j=1}^M T_{i,j} V_j y\rangle = \langle x, {\mathcal V}^* \pi(A) \pi'(T^{\rm T}) {\mathcal V} y\rangle\ .
 $$
 Defining $\widetilde{T} := \pi'(T^{\rm T})$, we obtain \eqref{arv22}.  
 
 Finally, assume that $\Psi$ has the form \eqref{arv22}.   Since $\pi(\cA)' = \pi'(M_M(\C))$ we can write $\widetilde{T}= \pi'(T^{\rm T})$ for some $M\times M$ matrix $T$ with $0 \leq T \leq 1$. Then
 $$
  \Psi(A) = {\mathcal V}^* \pi(A) \widetilde{T} {\mathcal V} =  \Psi(A) = {\mathcal V}^* \pi(A) \pi'(T^{\rm T})  {\mathcal V} = {\mathcal V}^* \pi(A)   {\mathcal W} 
  $$
  where
  $$
  {\mathcal W} x = [W_1x,\dots, W_Mx] \qquad{\rm and}\qquad W_i = \sum_{j=1}^M T_{i,j}V_j\ .
  $$
  Then $\Psi$ has the form specified in \eqref{arv21}
 \end{proof}

 We now turn to a question addressed by Arveson:  A CP map is {\em unital} in case $\Phi(\one) = \one$. Evidently the set of unital CP maps is  convex. An element $\Phi$ of this convex set is extremal in case whenver $\Psi$ is another unital CP map
 such that  for some $t\in (0,1)$, $t\Psi \leq \Phi$, then necessarily $\Psi =  \Phi$.  What are necessary and sufficient conditions for a unital CP map $\Phi$ to be extremal in the cone of unital CP maps?  
 
 Arveson's answer is stated in terms of minimal Stinespring representations and the commutant of $\pi({\mathcal M}_N(\C))$, where $\pi$ is the representation in the Stinespring representation. In our finite dimensional setting, there is a much simpler expression, due to Choi,  of this condition in terms of a minimal Kraus representation:
 
  \begin{thm}  Let $\Phi$ be a unital  CP map with a minimal Kraus representation  $\Phi(A) = \sum_{j=1}^M V_j^*A V_j$.  In order for $\Phi$ to be extremal in the cone of unital CP maps, it is necessary and sufficent that  the $M^2$ matrices
 \begin{equation}\label{lininset}
 \{\  V_i^*V_j\ :\ 1 \leq i,j \leq M \ \}
 \end{equation}
 are linearly independent.
 \end{thm}

 \begin{remark} Choi \cite[Theorem 5]{Choi75} gave an elementary proof of this result that bypasses the use of Arveson's Radon-Nikodym Theorem. We give a very short matricial rendering of Arveson's original proof \cite[Theorem 1.4.6]{Arv69} using Theorem~\ref{arveRN}.  It is worth noting, as Arveson did,  that the proof applies, and yields the same result, if applied to the class of CP maps for which $\Phi(\one) = K$, for any fixed $0 \leq K \leq \one$. 
 \end{remark}

\begin{proof} Suppose first that the set in \eqref{lininset} is linearly independent.  Let  $\Psi$  be CP and unital with $t\Psi \leq \Phi$ for some $0 < t < 1$.  Then by Theorem~\ref{arveRN}, there is an $M\times M$ matrix $T$ with $0 \leq T \leq \one$ such that 
$$t\Psi (A) = \sum_{i,j=1}^M T_{i,j} V_i^* A V_j\ .$$
Taking $A= \one$, we get
${\displaystyle t\one = \sum_{i,j=1}^M T_{i,j} V_i^* V_j}$, and then since $t\one =  t\Phi(\one) = \sum_{j=1}^M tV_j^* V_j$,
we have  $\sum_{i,j=1}^M B_{i,j} V_i^* V_j = 0$  where $B_{i,j}  = t\delta_{i,j} - T_{i,j}$. By the linear independence, $B_{i,j} =0$ for each $i,j$, and hence 
$T = t\one$. Thus $t\Psi = t\Phi$, and so $\Phi$ is extreme. 

For the converse, suppose that $\Phi$ is extreme. The set in \eqref{lininset} is linearly independent if and only if the map $B \mapsto \sum_{i,j} B_{i,j}V_i^*V_j$  is injective. Because this map  is Hermitian, to show that it is injective, it suffices to show that  it is injective on the self-adjoint $M\times M$ matrices.  Therefore, consider a self-adjoint $B$ such that  $\sum_{i,j} B_{i,j}V_i^*V_j= 0$.  Replacing $B$ with a positive multiple of itself, we may freely assume that $\|B\| \leq 1$. Define $T = \tfrac12(\one + B)$,
and then define $\Psi$ by $\Psi(A) =  \sum_{i,j=1}^M T_{i,j} V_i^* V_j$.   $\Psi$ is $CP$ with $\Psi \leq \Phi$ by Theorem~\ref{arveRN}, and $\Psi(\one) = \tfrac12\one$ since $\sum_{i,j} B_{i,j}V_i^*V_j= 0$.  Therefore, defining $\widetilde\Psi = 2\Psi$, we have that $\widetilde \Psi$ is unital, CP, and $\frac12 \widetilde\Psi \leq \Phi$. Since $\Phi$ is extreme, $\widetilde\Phi = \Phi$, and hence $2T = \one$ and $B = 0$. 
\end{proof}

 \section{Acknowledgements} The work of both authors was partially supported by NSF grants DMS 1501007 and DMS 1764254.


\begin{thebibliography}{30}

\footnotesize{




\bibitem{A78} R.~Alicki, \textit{On the detailed balance condition for non-Hamiltonian systems}, Rep. Math. Phys., {\bf 10} 249-258.
1976 

\bibitem{Arv69} W.~B.~Arveson, \textit{Subalgebras of $C^*$ algebras}, Acta Math.. {\bf 123} (1969), 141-224.

\bibitem{CM17} E.~A.~Carlen and J.~Maas, 
	\textit{Gradient flow and entropy inequalities for quantum Markov semigroups with detailed balance}, 
	   Jour.  Func. Analysis, {\bf 273}, no. 5, (2017) 1810-1869
	   
\bibitem{CM20} E.A. Carlen and J. Maas, 
	{\it Non-commutative calculus, optimal transport and functional inequalities in dissipative quantum systems},  J. Stat. Phys. {\bf 178} (2020) 319-378.

\bibitem{Choi75} M.~D.~Choi,  \textit{Completely positive linear maps on complex matrices}, Lin. Alg. and Appl. {\bf 10}  285-290 1975.

\bibitem{FU07} F.~Fagnola and V.~Umanit\`a, \textit{Generators of Detailed Balance Quantum Markov Semigroups}, Infin. Dimens. Anal. Quantum Probab. Relat. Top.,
{\bf 10}, no. 3, 335-363, 2007.

\bibitem{FU10} F.~Fagnola and V.~Umanit\`a, \textit{Generators of KMS Symmetric Markov Semigroups on B(h): Symmetry and Quantum Detailed Balance}. 
	Commun.	Math. Phys. {\bf 298},  523-547, 2010.


\bibitem{GKS76} V.~Gorini, A.~Kossakowski and E.~C.~G.~Sudarshan, \textit{Completely positive 
dynamical semigroups of $N$-level systems},
J. Math. Phys. {\bf 17}, 821-825,1976.



\bibitem{K71} K.~Kraus, \textit{General state changes in quantum theory}. Ann. Phys. {\bf 64} (1971), 311-335.




\bibitem{Lin76} G.~Lindblad, \textit{On the generators of quantum dynamical semigroups}, Comm. Math. Phys., 
{\bf 48}, 119-130, 1976.

\bibitem{PZ77} G.~Parravicini and A.~Zecca, \textit{On the generator cmpletely positive quantum dynamical semigroups of $N$-level systems}, Rep. Math. Phys., 
{\bf 12}, 423-424, 1977.


\bibitem{Paul02}
{V.~Paulsen}, {\em Completely bounded maps and operator algebras}, vol.~78
  of Cambridge Studies in Advanced Mathematics, Cambridge University Press,
  Cambridge, 2002.


\bibitem{Sti55} W.~F.~Stinespring, \textit{Positive functions on $C^*$ algebras}. Proc. Am. Math. Soc. {\bf 6}, 211-216
(1955)

\bibitem{Sto10} E.~St\o rmer, \textit{Positive linear maps of operator algebras}, Springer, Heidelberg, (2010).

}



\end{thebibliography}
\end{document}